\documentclass[11pt,twoside]{article}
\usepackage{amsfonts}
\usepackage{amssymb}

\usepackage{latexsym}
\usepackage{graphicx}
\usepackage{amsmath}
\usepackage{hyperref}
\numberwithin{equation}{section}

\normalsize \makeatletter
\newenvironment{proof}{{\em Proof.} }{\hfill$\Box$\vspace{0.1in}}

\newtheorem{definition}{Definition}[section]
\newtheorem{proposition}{Proposition}[section]

\newtheorem{lemma}{Lemma}[section]

\newtheorem{remark}{Remark}[section]
\newtheorem{theorem}{Theorem}[section]
\newtheorem{assumption}{Assumption}[section]

\newcommand{\be}{\begin{equation}}
\newcommand{\ee}{\end{equation}}
\newcommand{\bea}{\begin{eqnarray}}
\newcommand{\eea}{\end{eqnarray}}

\newtheorem{thm}{Theorem}[section]

\newtheorem{prop}[thm]{Proposition}

\begin{document}

\title{ The nonlinear Schr\"odinger equations with combined nonlinearities of  power-type and
Hartree-type \thanks {This work is supported by NSFC 10571158. }}

\author{ Daoyuan Fang and  Zheng Han \thanks {email: DF:
dyf@zju.edu.cn, HZ: hanzheng5400@yahoo.com.cn}\\
Department of Mathematics, Zhejiang University,\\ Hangzhou 310027,
China,}

 \maketitle
\begin{abstract}
This paper is devoted to a comprehensive study of the nonlinear
Schr\"odinger equations with combined nonlinearities of the
power-type and  Hartree-type
 in any dimension $n\ge 3$. With some structural
conditions, a nearly whole picture of the interactions of these
nonlinearities  in the energy space is given. The method is based on
the Morawetz estimates and
perturbation principles.
\end{abstract}

\section{Introduction}
\ \quad We are concerned with the Cauchy problem for the following
Schr\"odinger equation
\begin{eqnarray}\label{1}
 \left\{
\begin{array}{ll}
iu_t+\Delta u=\lambda_1|u|^pu+\lambda_2(|x|^{-\gamma}\ast
|u|^2)u\\
u(0,x)=u_0(x),
\end{array}
\right.
\end{eqnarray}
where $u(t,x)$ is a complex-value function in spacetime
$\mathbb{R}\times\mathbb{R}^n(n\geq3)$, initial datum $u_0$ takes
value in $H_x^1(\mathbb{R}^n)$(or $\sum=\left\{u\in
H_x^1(\mathbb{R}^n):|\cdot|u(\cdot)\in
L_x^2(\mathbb{R}^n)\right\}$), $\lambda_1$ and $ \lambda_2$ are
nonzero constants, $0<p\leq\frac{4}{n-2}$, and $ 0<\gamma\leq4$ with
$n>\gamma$. For such a problem, T. Cazenave has given a fundamental
discussion  in \cite{19}. However, just a few cases he has settled,
for example when both nonlinearities are defocusing, the equation
must be energy-subcritical; when one nonlinearity is focusing, and
the index of the nonlinearity lies between mass-critical and
energy-critical, need mass and energy sufficiently small. In very
recent years, there are many results on the global well-posedness
for the following energy-critical \eqref{2} and \eqref{4} or
mass-critical \eqref{3} and \eqref{5} nonlinear Schr\"odinger
equation have been obtained by T. Tao, J. Colliander and Carlos E.
Kenig and so on, respctively. \cite{2,3,4,7,10,11,15,17, 18}
\begin{eqnarray}\label{2}
 \left\{
\begin{array}{ll}
iu_t+\Delta u=\lambda_2(|x|^{-4}\ast
|u|^2)u\\
u(0,x)=u_0(x)
\end{array}
\right.
\end{eqnarray}
\begin{eqnarray}\label{4}
 \left\{
\begin{array}{ll}
iu_t+\Delta u=\lambda_1|u|^{\frac{4}{n-2}}u\\
u(0,x)=u_0(x)
\end{array}
\right.
\end{eqnarray}

\begin{eqnarray}\label{3}
 \left\{
\begin{array}{ll}
iu_t+\Delta u=\lambda_2(|x|^{-2}\ast
|u|^2)u\\
u(0,x)=u_0(x)
\end{array}
\right.
\end{eqnarray}
\begin{eqnarray}\label{5}
 \left\{
\begin{array}{ll}
iu_t+\Delta u=\lambda_1|u|^{\frac{4}{n}}u\\
u(0,x)=u_0(x)
\end{array}
\right.
\end{eqnarray}
Therefore,  in this paper, we want to give a whole picture of the
interactions of these both nonlinearities. First of all, we hope to
solve the same problem of \eqref{1} when one nonlinearity is
energy-critical. Then, we discuss the case
$\lambda_1\cdot\lambda_2<0$ that T. Cazenave didn't take care of but
separately. Precisely, we hope that under some structural
conditions, that is, under some relations of $\lambda$ and $p$, the
defocusing term is able to control the focusing term, so that the
whole nonlinearities behaviour like the defocusing property,
therefore there is a global wellposed behaviour to be appeared
because the defocusing nonlinearity will amplify the dispersive
effect of the linear equation, but the focusing one usually is to
cancel the dispersive effect.\par
 Schr\"{o}dinger equation \eqref{1} has two conservation laws: energy conservation and mass conservation, where energy
 and mass are defined as follow:
 \begin{eqnarray*}
E(u(t)):&=&\frac{1}{2}\int|\nabla
u|^2\,dx+\frac{\lambda_1}{p+2}\int|u|^{p+2}\,dx+\frac{\lambda_2}{4}\int\left(|x|^{-\gamma}\ast|u|^2\right)|u|^2\,dx\\
M(u(t)):&=&\int|
u|^2\,dx
 \end{eqnarray*}
As they are conserved, we'll prefer to write $E(u)$ for $E(u(t))$
and $M(u)$ for $M(u(t))$.\par
 Our first main theorem is  as following:
\begin{theorem}(Global well-posedness)\label{1.1}
Let $u_0\in H_x^1$. Then there exists a unique global solution u to
\eqref{1} in each of the following cases:
\begin{enumerate}
  \item \label{case 1} when $\lambda_1,\lambda_2>0,\
        0<p\leq\frac{4}{n-2},\ 0<\gamma\leq4$ with $\gamma<n$ except for
        $(p,\gamma)=(\frac{4}{n-2},4)$.
  \item \label{case 2} when $\lambda_1>0,\ \lambda_2<0,$
        \begin{description}
          \item[2.1]  $0<p\leq\frac{4}{n-2}$, and
          $0<\gamma<\min{\{n,\frac{np}{2}\}}$.
          \item[2.2] $\frac{np}{2}\leq\gamma<2$.
          \item[2.3]  $\frac{np}{2}\leq\gamma=2,$ and $ \parallel u_0\parallel_{L^2}^2<\frac{1}{|\lambda_2|}\parallel
          W\parallel_{L^2}^2$.
          \item[2.4]  $\frac{np}{2}\leq\gamma=4\ (n>4),\ E<\frac{\tilde{E}(W)}{|\lambda_2|},\ \parallel \nabla u_0\parallel_{L^2}^2<\frac{1}{|\lambda_2|}\parallel \nabla
          W\parallel_{L^2}^2$ and $u_0$ is radial except for
          $(p,\gamma)=(\frac{4}{n-2},4)$.
          \item[2.5] $\frac{np}{2}\leq\gamma,\ 2<\gamma<\min{\{4,n\}},\
          EM^\frac{4-\gamma}{\gamma-2}<(\frac{1}{2}-\frac{1}{\gamma})\left[\frac{2\gamma\tilde{E}(W)}{|\lambda_2|(\gamma-2)}\right]^{\frac{2}{\gamma-2}}$
          \\and $\parallel \nabla u_0\parallel_{L^2}^{2}M^\frac{4-\gamma}{\gamma-2}<\left(\frac{\parallel\nabla
          W\parallel_{L^2}^2}{|\lambda_2|}\right)^{\frac{2}{\gamma-2}}$,
        \end{description}
        where W is the solution of ground state:
          $W+\left(|x|^{-\gamma}\ast|w|^2\right)W=\frac{4-\gamma}{\gamma}W$\\
          and $\tilde{E}(W):=\frac{1}{2}\int|\nabla
          W|^2\,dx-\frac{1}{4}\int\left(|x|^{-\gamma}\ast|w|^2\right)|W|^2\,dx$.
  \item \label{case 3} when $\lambda_1<0,\ \lambda_2>0$,
        \begin{description}
          \item[3.1]  $0<p<\max{\{\frac{4}{n},
          \frac{4}{2+n-\gamma}\}},$ and $
          0<\gamma\leq4$ with $\gamma<n$.
          \item[3.2]  $p=\frac{4}{n},\ p\geq\frac{4}{2+n-\gamma},$ and $ \parallel u_0\parallel_{L^2}<|\lambda_1|^{-\frac{n}{4}}\parallel
          R\parallel_{L^2}$.
          \item[3.3]  $\frac{4}{2+n-\gamma}\leq
          p=\frac{4}{n-2}$ except for
        $(p,\gamma)=(\frac{4}{n-2},4)$, in addition, \\
          if $n\geq5,$ require $E<|\lambda_1|^{\frac{2-n}{2}}\tilde{E}(R),
         \parallel \nabla u_0\parallel_{L^2}^2<|\lambda_1|^{\frac{2-n}{2}}\parallel
          \nabla R\parallel_{L^2}^2$, \\
          if $n=3,4$,  $u_0$ is radial.
          \item[3.4] $\frac{4}{n}<p<\frac{4}{n-2}$, and $ \frac{4}{2+n-\gamma}\leq
          p$ with\\
          \begin{eqnarray*}
          EM^\frac{4-(n-2)p}{np-4}&<&|\lambda_1|^{\frac{4}{4-np}}\left(\frac{2np}{np-4}\right)^{\frac{4-(n-2)p}{np-4}}
          \left(\tilde{E}(R)\right)^\frac{2p}{np-4},\\
          \parallel \nabla u_0\parallel_{L^2}^{2}M^\frac{4-(n-2)p}{np-4}&<&|\lambda_1|^{\frac{4}{4-np}}\parallel
          \nabla R\parallel_{L^2}^{\frac{4p}{np-4}},
          \end{eqnarray*}
          \end{description}
          where R is the solution of ground state:
          $\Delta R+|R|^pR=\frac{4-(n-2)p}{np}R$\\
          and $\tilde{E}(R):=\frac{1}{2}\int|\nabla
          R|^2\,dx-\frac{1}{p+2}\int|R|^{p+2}\,dx$.
  \item \label{case 4} $\lambda_1<0,\ \lambda_2<0,\
  0<p<\frac{4}{n},$ and $
  0<\gamma<2$.
\end{enumerate}
Moreover, for all compact intervals $I$, the global solution
satisfies the following spacetime bound:
\begin{equation}
\parallel u\parallel_{S^1(I\times\mathbb{R}^n)}\leq
C(|I|,E,M).\label{1.6}
\end{equation}
\end{theorem}
\begin{remark}
For the case 2.4, we need the initial datum to be radial. Because
according to \cite{4} when the initial datum is radial, there maybe
exists the global solution for \eqref{2}. For the case 3.3, R.
Killip and M. Visan have proven the global well-posedness for
\eqref{4} in \cite{17} when the initial datum isn't radial, but
their approach is not suitable for the lower dimension, thus for
lower dimension we preserve the radial condition.
\end{remark}
\ \quad We'll prove this theorem in Section 4. Our chief work is to
get a bound of $\parallel u\parallel_{H_x^1}$ which only depends on
energy and mass, and then apply the perturbation principles to get
the result. As mentioned above, we hope the defocusing term can
control the focusing term, however, this can't be true usually, but
we can prove that under the assumption of 2.1 and 3.1 in Theorem
1.1, it do happen. For other cases, our approach can't show the
defocusing term is able to control the focusing term. So just as
what T. Cazenave did, still need some circumstances of the smallness
about energy and mass. But the different point from that is the
smallness which is characterized by the ground state. Unfortunately,
our method isn't useful for the case that both of the power and
Hartree nonlinearities are energy-critical. Because after using
Strichartz estimate, we need the dependence in time for the
coefficients of nonlinearities, but no such factor for such both
cases are energy-critical. The detail is in Section 4.

 In Section 5, we consider the asymptotic behavior of these global solutions.
 It is natural to apply a unconditional scattering theory for
\eqref{3} and \eqref{5}. However, at least till now, we have to
demand the initial datum radial and the size of mass is smaller than
the one of ground state \cite{15,18}. Therefore, we need the
following assumptions:
\begin{assumption}
Let $v_0\in H_x^1,\ \lambda_1>0$. Then there exists a unique global
solution $v$ to \eqref{5} and satisfies
\begin{equation}
\parallel v\parallel_{L^\frac{2(n+2)}{n}_{t,x}(\mathbb{R}\times
\mathbb{R}^n)}\leq C(\parallel v_0\parallel_{L_x^2})
\end{equation}
\end{assumption}
\begin{assumption}
Let $w_0\in H_x^1,\ \lambda_2>0$. Then there exists a unique global
solution $w$ to \eqref{3} and moreover
\begin{equation}
\parallel w\parallel_{L^6_tL^\frac{6n}{3n-2}_x(\mathbb{R}\times
\mathbb{R}^n)}\leq C(\parallel w_0\parallel_{L_x^2})
\end{equation}
\end{assumption}
\ \quad Our second main theorem is:
\begin{theorem}\label{thm1.2}(Energy space scattering)\\
Let $u_0\in H_x^1$, the conditions in Theorem 1.1  are assumed, and
$u$ be the unique solution to \eqref{1}. In addition, if
$p=\frac{4}{n}$, then we need Assumption 1.1. If $\gamma=2$, then we
also need Assumption 1.2. Then in the following case, there exist
$u_{+}, u_-\in H_x^1$ such that
\begin{equation}
\parallel
u-e^{it\Delta}u_{\pm}\parallel_{H_x^1}\rightarrow0\quad \mbox{as}\
t\rightarrow\pm\infty.
\end{equation}
case 1: $\lambda_1,\ \lambda_2>0,\ \frac{4}{n}\leq
p\leq\frac{4}{n-2},\ 2\leq\gamma\leq4$ with $\gamma<n$ except the
point
$(p,\gamma)=(\frac{4}{n-2},4)$, especially, when $(p,\gamma)=(\frac{4}{n},2)$, we still need the small mass condition;\\
case 2: $\lambda_1\cdot\lambda_2<0,\ \frac{4}{n}\leq
p\leq\frac{4}{n-2},\ 2\leq\gamma\leq4$ with $\gamma<n$ and the small
mass condition except the point $(p,\gamma)=(\frac{4}{n-2},4)$.\\
Furthermore,
\[
\parallel u_+\parallel_{L^2}=\parallel u_-\parallel_{L^2}=\parallel
u_0\parallel_{L^2}\quad\mbox{and}\quad\frac{1}{2}\int_{\mathbb{R}^n}|\triangledown
u_+|^2=\frac{1}{2}\int_{\mathbb{R}^n}|\triangledown u_-|^2=E(u_0)
\]
\end{theorem}
\ \quad We'll prove the theorem in Section 5. The main tools are a
refined Morawetz estimate and the perturbation principles.
Unfortunately, to use such a refined Morawetz estimate, we have to
require that $\lambda_1>0,\ \lambda_2>0,\ p>\frac{4}{n}\ ,\gamma>2$.
So when $\lambda_1\cdot\lambda_2<0$ need a kind of smallness, here
we demand of mass sufficiently small. The refined Morawetz estimate
was firstly used by T. Tao to prove the dispersive property of the
cubic Schr\"odinger equation \cite{12}, but the space dimension must
be no less than 3. Then, J. Colliander, M. Grillakis and N. Tzirakis
get a refined Morawetz estimate for 1-D and 2-D, and obtain the
scattering of 2-D power type Schr\"odinger equation. However, for
this case $\gamma<n=2$, in order to apply Morawetz estimate, we need
$\gamma>2$. Thus we can't have scattering for Hartree, as well for
\eqref{1}. For $p=\frac{4}{n}$ and $\gamma=2$, i.e. both
nonlinearities are mass-critical, the low frequency of the solution
can own an effective control, but  there is no such a  good luck for
the high one. Thus at this time, we view \eqref{1} as the
perturbation of free Schr\"odinger equation.\par
 At the last section, we describe the blow up phenomena when the initial datum belongs to $\Sigma$ space. We believe the method also suitable for
the initial datum belongs to energy space with radial condition. The
detail can be consulted in Chapter 6 of \cite{19}.\par
 Our last main theorem is:
\begin{theorem}(blowup)\\
Let $u_0\in \Sigma$. Then blowup occurs in each of the following
cases: \\
(1) for $\lambda_1>0,\ \lambda_2<0$: when $2\leq\gamma\leq4,\
0<p\leq\frac{4}{n-2}\ ,\gamma\geq\frac{np}{2},\ \mbox{and}\ E<0$;\\
(2) for $\lambda_1<0,\ \lambda_2>0$: when $\frac{4}{n}\leq p\leq\frac{4}{n-2},\ 0<\gamma\leq\frac{np}{2},\ \mbox{and}\ E<0$;\\
(3) for $\lambda_1<0,\ \lambda_2<0$
    \begin{itemize}
      \item when $\frac{4}{n}<p\leq\frac{4}{n-2},\ 0<\gamma<2,\
      \mbox{and}\ 4npE+C(M)<0$.
      \item when $0<p<\frac{4}{n},\ 2<\gamma\leq4,\
      \mbox{and}\ 8\gamma E+C(M)<0$.
      \item when $\frac{4}{n}\leq p\leq\frac{4}{n-2},\ 2\leq \gamma\leq4,\
      \mbox{and}\ E<0$.
    \end{itemize}
\end{theorem}
\begin{remark}
The conclusions in Theorem 1.1 and Theorem 1.3 aren't contrary,
since the energy in Theorem 1.1 are nonnegative. One can find that
for the case $\lambda_1<0,\ \lambda_2>0$, we drop a situation:
$\frac{np}{2}<\gamma\leq2+n-\frac{4}{p}$, it was caused by that we
could not judge the relationship between
$\int\left(|x|^{-\gamma}\ast|u|^2\right)|u|^2\,dx$ and $\parallel
u\parallel_{L_x^{p+2}}^{p+2}$.  Since the inequality
\[\parallel
u\parallel_{L_x^q}^q\lesssim\int\left(|x|^{-\gamma}\ast|u|^2\right)|u|^2\,dx,\,
\parallel
u\parallel_{L_x^{p+2}}^{p+2}\lesssim\parallel u\parallel_{L_x^r}^r
\]
holds true where $q=\frac{2(4+n-\gamma)}{2+n-\gamma},\
r=\frac{2n+2\gamma}{n}$. If we can get
$\int\left(|x|^{-\gamma}\ast|u|^2\right)|u|^2\,dx\thicksim\parallel
u\parallel_{L_x^s}^s$, for the situation $s>p+2$, one can apply the
method in Subsection 4.2 for case (2), to get the global
well-posedness and scattering; for the other $s\leq p+2$, one can
apply the method in Section 6, to say under some condition, it would
blow up in finite time.
\end{remark}
\section{Notation}
\ \quad In this section, we will introduce a few notations and
fundamental inequalities which always appear in the following
sections.
 \begin{definition}: We say a pair $(q,r)$ is
Schr\"{o}dinger-admissible if $\frac{2}{q}+\frac{n}{r}=\frac{n}{2}$
and $2\leq q,r\leq\infty$. If $I\times \mathbb{R}^n$ is a spacetime
slab, we define:
\[
||u||_{\dot{S}^0(I\times\mathbb{R}^n)}:=\sup||u||_{L^q_tL^r_x(I\times
\mathbb{R}^n)},
\]
where the {\rm sup} is taken over all admissible pairs $(q,r)$,
\[
||u||_{\dot{S}^1(I\times \mathbb{R}^n)}:=||\nabla
u||_{\dot{S}^0(I\times \mathbb{R}^n)}.
\]
Denote $\dot{N}^0(I\times \mathbb{R}^n)$ the dual space of
$\dot{S}^0(I\times \mathbb{R}^n)$, and
\[
\dot{N}^1(I\times\mathbb{R}^n):=\{u:\nabla u\in\dot{N}^0(I\times
\mathbb{R}^n)\}.
\]
We also define the following norms:
\begin{eqnarray*}
||u||_{U(I)}&:=&||u||_{L^6_tL^\frac{6n}{3n-2}_x(I\times
\mathbb{R}^n)}\\
||u||_{V(I)}&:=&||u||_{L^\frac{2(n+2)}{n}_{t,x}(I\times
\mathbb{R}^n)}\\
||u||_{W(I)}&:=&||u||_{L^\frac{2(n+2)}{n-2}_{t,x}(I\times
\mathbb{R}^n)}\\
||u||_{Z(I)}&:=&||u||_{L^{n+1}_tL^\frac{2(n+1)}{n-1}_x(I\times
\mathbb{R}^n)}
\end{eqnarray*}
\end{definition}
By definition and Sobolev embedding, we obtain
 \begin{lemma} For
any $\dot{S}^1$ function u on $I\times \mathbb{R}^n$, we have
\begin{eqnarray}
||\nabla u||_{L_t^{\infty}L_x^{2}}+||\nabla
u||_{L_t^{\frac{2(n+2)}{n-2}}L_x^{\frac{2n(n+2)}{n^2+4}}}+||\nabla
u||_{V}+||\nabla u||_{L_t^{2}L_x^{\frac{2n}{n-2}}}+||\nabla
u||_{U}\nonumber\\
+||u||_{L_t^{\infty}L_x^{\frac{2n}{n-2}}}
+||u||_{W}+||u||_{L_t^{\frac{2(n+2)}{n}}L_x^{\frac{2n(n+2)}{n^2-2n-4}}}\lesssim
||u||_{\dot{S}^1},
\end{eqnarray}
where all spacetime norms are on $I\times \mathbb{R}^n$.
\end{lemma}
\begin{lemma}(\textbf{Strichartz estimates})Let $I$ be a compact time interval,
$k=0,1$, and  $u:I\times \mathbb{R}^n\rightarrow \mathbb{C}$ be an
$\dot{S}^k$ solution to the forced Schr\"{o}dinger equation
\[
iu_t+\Delta u=F
\]
for a given function $F$. Then we have
 \begin{eqnarray}
||u||_{\dot{S}^k(I\times
\mathbb{R}^n)}\lesssim
||u(t_0)||_{\dot{H}^k(\mathbb{R}^n)}+||F||_{\dot{N}^k(I\times
\mathbb{R}^n)}
\end{eqnarray}
for any time $t_0\in I$.
\end{lemma}\par
 For the details of proof, we refer to \cite{16,19}. In addition, we need some Littlewood-Paley theory. Let $\varphi(\xi)$
 be a smooth bump function which is supported in the ball $|\xi|\leq 2$
 and equal to 1 in the ball $|\xi|\leq 1$. For each dyadic
 number $N\in 2^\mathbb{Z}$, we can define the Littlewood-Paley
 operators:
 \begin{eqnarray*}
 \widehat{P_{\leq
 N}f}(\xi):&=&\varphi(\frac{\xi}{N})\hat{f}(\xi),\\
 \widehat{P_{>
 N}f}(\xi):&=&[1-\varphi(\frac{\xi}{N})]\hat{f}(\xi),\\
 \widehat{P_{
 N}f}(\xi):&=&[\varphi(\frac{\xi}{N})-\varphi(\frac{2\xi}{N})]\hat{f}(\xi).
 \end{eqnarray*}
Then by these notations, we recall a few standard Bernstein type
inequalities:
 \begin{lemma} For any $1\leq p\leq q\leq\infty, s>0$, we have
 \begin{eqnarray*}
 ||P_{\geq N}f||_{L^p_x}&\lesssim&N^{-s}|||\nabla|^sP_{\geq
 N}f||_{L^p_x},\\
 |||\nabla|^sP_{\leq N}f||_{L^p_x}&\lesssim&N^{s}||P_{\leq
 N}f||_{L^p_x},\\
 |||\nabla|^{\pm s}P_{N}f||_{L^p_x}&\thicksim&N^{\pm s}||P_{
 N}f||_{L^p_x},\\
 ||P_{\leq N}f||_{L^q_x}&\lesssim&N^{\frac{n}{p}-\frac{n}{q}}||P_{\leq
 N}f||_{L^p_x},\\
 ||P_{N}f||_{L^q_x}&\lesssim &N^{\frac{n}{p}-\frac{n}{q}}||P_{
 N}f||_{L^p_x}.
 \end{eqnarray*}
 \end{lemma}
 \begin{definition} Let $I\times
\mathbb{R}^n$ be an arbitrary spacetime slab, we
 define the space
 \begin{eqnarray*}
 \dot{X}^0(I)=\left\{\begin{array}{ll}
                    L^q_tL^r_x(I\times
\mathbb{R}^n),\qquad& 0<p<\frac{4}{n-2},\\
                    L_t^{\frac{2(n+2)}{n-2}}L_x^{\frac{2n(n+2)}{n^2+4}}(I\times
\mathbb{R}^n)\cap V(I),\qquad& p=\frac{4}{n-2},
                     \end{array}
                     \right.
 \end{eqnarray*}
 where\ $q=\frac{4(p+2)}{p(n-2)}, r=\frac{n(p+2)}{n+p}$,\\
 and
 \begin{eqnarray*}
 &\dot{X}^1:=\{u: \nabla u\in\dot{X}^0(I)\},\qquad
 X^1(I):=\dot{X}^0(I)\cap\dot{X}^1(I),&\\
 \\
 &\dot{Y}^0(I):=\left\{\begin{array}{ll}
                    L^\infty_tL^2_x(I\times
\mathbb{R}^n),\qquad& 0<\gamma\leq 2,\\
                    L_t^\infty L_x^2(I\times
\mathbb{R}^n)\cap L_t^\mu L_x^\sigma(I\times \mathbb{R}^n), \qquad&
2<\gamma\leq 4\  and\  \gamma<n
                     \end{array}
                     \right.,&
 \end{eqnarray*}
 where\ $\mu=\frac{6}{\gamma-2},\sigma=\frac{6n}{3n+4-2\gamma}$,\\
 and
 \begin{eqnarray*}
 \dot{Y}^1:&=&\{u: \nabla u\in\dot{Y}^0(I)\},\qquad
 Y^1(I):=\dot{Y}^0(I)\cap\dot{Y}^1(I),\\
 \dot{B}^0(I):&=&\dot{X}^0(I)\cap\dot{Y}^0(I),\qquad\dot{B}^1:=\{u: \nabla
 u\in\dot{B}^0(I)\},\qquad
 B^1(I):=\dot{B}^0(I)\cap\dot{B}^1(I).
 \end{eqnarray*}
 \end{definition}\par
 Furthermore, we also need the following maximal estimate which
 is a direct con-sequence of the sharp Hardy inequality \cite{21}.
 \begin{lemma} Let $0<\gamma<n$, we have
 \begin{eqnarray}
 \parallel |x|^{-\gamma}\ast|u|^2\parallel_{L_x^\infty}\leq
 C(n,\gamma)||u||_{\dot{H}^{\frac{\gamma}{2}}}^2.
 \end{eqnarray}
 \end{lemma}
 \begin{lemma} Let I be a compact time interval, $0<p\leq\frac{4}{n-2},0<\gamma\leq
 4$ and $\gamma<n, \lambda_1$ and $\lambda_2$ be nonzero real numbers,
 and $k=0,1$. Then
 \begin{eqnarray}
&&\parallel\lambda_1|u|^pu+\lambda_2(|x|^{-\gamma}\ast|u|^2)u\parallel_{\dot{N}^k(I\times\mathbb{R}^n)}\nonumber\\
&&\lesssim|I|^{1-\frac{p(n-2)}{4}}\parallel
 u\parallel^p_{\dot{X}^1(I)}\parallel
 u\parallel_{\dot{X}^k(I)}+|I|^\alpha\parallel
 u\parallel^2_{\dot{Y}^1(I)}\parallel u\parallel_{\dot{Y}^k(I)}\label{2.4}\\
&&\parallel\left(\lambda_1|u|^pu+\lambda_2(|x|^{-\gamma}\ast|u|^2)u\right)-
 \left(\lambda_1|v|^pv+\lambda_2(|x|^{-\gamma}\ast|v|^2)v\right)\parallel_{\dot{N}^0(I\times\mathbb{R}^n)}\nonumber\\
&& \lesssim
 |I|^{1-\frac{p(n-2)}{4}}\left(\parallel
 u\parallel^p_{\dot{X}^1(I)}+\parallel
 v\parallel^p_{\dot{X}^1(I)}\right)\parallel u-v\parallel_{\dot{X}^0(I)}+|I|^\alpha\left(\parallel
 u\parallel^2_{\dot{Y}^1(I)}+\parallel
 v\parallel^2_{\dot{Y}^1(I)}\right)\parallel u-v\parallel_{\dot{Y}^0(I)},\nonumber\\
 &&\\
 &&where\quad
  \alpha=\left\{\begin{array}{ll}
                  1\qquad& 0<\gamma\leq 2\\
                  2-\frac{\gamma}{2}\ & 2<\gamma\leq
                  4\ \mbox{and}\
                  \gamma<n.
                  \end{array}
                  \right.
                  \nonumber
 \end{eqnarray}
 \end{lemma}
 \begin{proof}: Using H\"{o}lder, Sobolev embedding,
 Hardy-Littlewood-Sobolev inequality and Lemma 2.4, we can
 obtain the results.
 \end{proof}
 \begin{lemma}\label{lemma2.6}
 Let $k=0,1,\ \frac{4}{n}<p<\frac{4}{n-2}$ and $2<\gamma<\min\{4,n\}$.
 Then there exists $\theta>0$ large enough such that on each slab
 $I\times\mathbb{R}^n$, we have
 \begin{eqnarray}
 \parallel|u|^pu\parallel_{\dot{N}^k(I\times\mathbb{R}^n)}\lesssim
 \parallel u\parallel_{\dot{S}^k(I\times\mathbb{R}^n)}\parallel
 u\parallel_{Z(I)}^\frac{n+1}{2(2\theta+1)}\parallel u\parallel_{L_t^\infty
 L_x^2}^{\alpha_1(\theta)}\parallel u\parallel_{L_t^\infty
 L_x^\frac{2n}{n-2}}^{\alpha_2(\theta)},\label{2.6}\\
 \parallel(|x|^{-\gamma}\ast|u|^2)u\parallel_{\dot{N}^k(I\times\mathbb{R}^n)}\lesssim
 \parallel u\parallel_{\dot{S}^k(I\times\mathbb{R}^n)}\parallel
 u\parallel_{Z(I)}^\frac{n+1}{2(2\theta+1)}\parallel u\parallel_{L_t^\infty
 L_x^2}^{\beta_1(\theta)}\parallel u\parallel_{L_t^\infty
 L_x^\frac{2n}{n-2}}^{\beta_2(\theta)},\label{2.7}\end{eqnarray}
 where

 \begin{eqnarray*}
 \alpha_1(\theta)&=&p(1-\frac{n}{2})+\frac{8\theta+1}{2(2\theta+1)},\quad
 \alpha_2(\theta)=\frac{n}{2}\left(p-\frac{n+8\theta+2}{n(2\theta+1)}\right),\\
 \beta_1(\theta)&=&(3-\gamma)+\frac{4\theta-1}{2(2\theta+1)},\qquad
 \beta_2(\theta)=(\gamma-1)-\frac{4\theta+n}{2(2\theta+1)}.
 \end{eqnarray*}
 \end{lemma}
 \begin{proof}
 For the former, one can find in \cite{24}. The same method
 can be used for the latter, we have
 \begin{eqnarray}
 \parallel(|x|^{-\gamma}\ast|u|^2)u\parallel_{\dot{N}^k(I\times\mathbb{R}^n)}&\lesssim&
 \parallel|\nabla|^k\left[(|x|^{-\gamma}\ast|u|^2)u\right]\parallel_{L_t^2L_x^\frac{2n}{n+2}(I\times\mathbb{R}^n)}\nonumber\\
 &\lesssim&\parallel|\nabla|^ku\parallel_{L_t^{2+\frac{1}{\theta}}L_x^{\frac{2n(2\theta+1)}{n(2\theta+1)-4\theta}}}
 \parallel u\parallel_{Z(I)}^{\frac{n+1}{2(2\theta+1)}}\parallel
 u\parallel_{L_t^\infty L_x^2}^{\beta_1(\theta)}\parallel
 u\parallel_{L_t^\infty
 L_x^{\frac{2n}{n-2}}}^{\beta_2(\theta)}\label{2.8}
 \end{eqnarray}
 which is obtained by using $H\ddot{o}lder$ and
 Hardy-Littlewood-Sobolev inequality, once $\beta_1(\theta)$ and
 $\beta_2(\theta)$ are positive.\\
 Note that
 $\left(2+\frac{1}{\theta},\frac{2n(2\theta+1)}{n(2\theta+1)-4\theta}\right)$
 is Schr\"{o}dinger-admissible. When $2<\gamma<4$, $\beta_1(\theta)$ and
 $\beta_2(\theta)$ will be positive if $\theta$ is large enough, because
 the above functions are increased in $\theta$, and
 when $\theta\rightarrow\infty$,
 \begin{eqnarray*}
 \beta_1(\theta)\rightarrow4-\gamma>0,\qquad
 \beta_2(\theta)\rightarrow\gamma-2>0.
 \end{eqnarray*}
 \end{proof}
 \begin{lemma}\label{lemma}
 Let $I\times\mathbb{R}^n$ be a spacetime slab. Then there exists a
 small constant $0<\rho<1$ such that
 \begin{eqnarray}
 \parallel|u|^\frac{4}{n-2}u\parallel_{\dot{N}^0(I\times\ast\mathbb{R}^n)}
 &\lesssim&\parallel u\parallel_{Z(I)}^\rho\parallel
 u\parallel_{S^1(I\times\ast\mathbb{R}^n)}^{\frac{n+2}{n-2}-\rho}\\
 \parallel(|x|^{-4}\ast|u|^2)u\parallel_{\dot{N}^0(I\times\mathbb{R}^n)}
 &\lesssim&\parallel(|x|^{-4}\ast|u|^2)u\parallel_{L_t^2L_x^\frac{2n}{n+2}(I\times\mathbb{R}^n)}\nonumber\\
 &\lesssim&\parallel
 u\parallel_{L_t^{2+\varepsilon}L_x^{\frac{2n}{n-2-\varepsilon}}}\parallel
 u\parallel_{Z(I)}^\rho\parallel
 u\parallel_{L_t^\infty L_x^2}^{\frac{\varepsilon(1+\varepsilon)}{2(2+\varepsilon)}}\parallel
 u\parallel_{L_t^\infty L_x^{\frac{2n}{n-2}}}^{2-\frac{\varepsilon(2+\varepsilon+n)}{2(2+\varepsilon)}}\nonumber\\
 &\lesssim&\parallel
 u\parallel_{Z(I)}^\rho\parallel
 u\parallel_{S^1(I\times\mathbb{R}^n)}^{3-\rho},\\
 \mbox{where}\
 \rho=\frac{\varepsilon(n+1)}{2(2+\varepsilon)}\ \mbox{and}\  \varepsilon\  \mbox{is a small
 constant}.\nonumber
\end{eqnarray}
\end{lemma}
 \begin{proof}
 The first result is proved in \cite{24}. For the other, note that
 $L_t^{2+\varepsilon}L_x^{\frac{2n}{n-2-\varepsilon}}$ interpolates between the
 $\dot{S}^0$-norm
 $L_t^{2+\varepsilon}L_x^{\frac{2n(2+\varepsilon)}{n(2+\varepsilon)-4}}$ and the $\dot{S}^1$-norm
 $L_t^{2+\varepsilon}L_x^{\frac{2n(2+\varepsilon)}{n(2+\varepsilon)-2(4+\varepsilon)}}$
 provided $\varepsilon$
 is sufficiently small, we have
 \[
 \parallel
 u\parallel_{L_t^{2+\varepsilon}L_x^{\frac{2n}{n-2-\varepsilon}}}\lesssim\parallel
 u\parallel_{S^1(I\times\mathbb{R}^n)}.
 \]
 Let $a(\varepsilon)=\frac{\varepsilon(1+\varepsilon)}{2(2+\varepsilon)},\ b(\varepsilon)=2-\frac{\varepsilon(n+2+\varepsilon)}{2(2+\varepsilon)}$,
   we only need to check
 $a(\varepsilon)$ and $ b(\varepsilon)$ are positive, since then the
 estimates is a simple consequence of H\"older inequality and
 Hardy-Littlewood-Sobolev inequality. As a function of $\varepsilon$,
 $a$ is increasing and $a(0)=0$, while $b$ is decreasing and $b(0)=2$.
 Thus, taking $\varepsilon>0$ sufficient small, we have $a(\varepsilon)>0,\
 b(\varepsilon)>0$. Taking
 $\rho=\frac{\varepsilon(n+1)}{2(2+\varepsilon)}$, we obtain the
 result.
 \end{proof}
 \begin{remark}\label{remark2.1}
 An easy consequence of the proof of Lemma 2.7 is that one can get the estimates for
 nonlinearities of the form $|u|^{\frac{4}{n-2}}v$ and
 $(|x|^{-\gamma}\ast|u|^2)v$. More precisely, we have
 \begin{eqnarray}
 \parallel
 |u|^{\frac{4}{n-2}}v\parallel_{\dot{N}^0(I\times\mathbb{R}^n)}
 &\lesssim&\parallel
 u\parallel_{Z(I)}^\rho\parallel
 u\parallel_{S^1(I\times\mathbb{R}^n)}^{\frac{4}{n-2}-\rho}\parallel
 v\parallel_{S^1(I\times\mathbb{R}^n)},\\
 \parallel(|x|^{-\gamma}\ast|u|^2)v\parallel_{\dot{N}^0(I\times\mathbb{R}^n)}
 &\lesssim&\parallel
 u\parallel_{Z(I)}^\rho\parallel
 u\parallel_{S^1(I\times\mathbb{R}^n)}^{2-\rho}\parallel
 v\parallel_{S^1(I\times\mathbb{R}^n)},\\
 \parallel\left(|x|^{-\gamma}\ast (wv)\right)v\parallel_{\dot{N}^0(I\times\mathbb{R}^n)}
 &\lesssim&\parallel
 u\parallel_{S^1(I\times\mathbb{R}^n)}\parallel
 w\parallel_{L_t^\infty L_x^2}^{a(\varepsilon)}\parallel
 v\parallel_{Z(I)}^\rho\parallel v\parallel_{L_t^\infty
 L_x^{\frac{2n}{n-2}}}^{b(\varepsilon)}.
 \end{eqnarray}
 \end{remark}
 \begin{lemma}\label{lemma2.8}
Let $I\times\mathbb{R}^n$ be an arbitrary spacetime slab,
$\frac{4}{n}\leq p\leq\frac{4}{n-2},\ 2\leq \gamma\leq 4$ with
$\gamma<n$, and $k=0,1$. Then
\begin{eqnarray*}
\parallel |u|^pu\parallel_{\dot{N}^k(I\times\mathbb{R}^n)}&\lesssim&\parallel u\parallel_{V(I)}^{2-\frac{(n-2)p}{2}}\parallel u\parallel_{W(I)}^{
\frac{np}{2}-2}\parallel|\nabla|^ku\parallel_{V(I)},\\
\parallel(|x|^{-\gamma}\ast|u|^2)u\parallel_{\dot{N}^k(I\times\mathbb{R}^n)}&\lesssim&\parallel u\parallel_{U(I)}^{4-\gamma}\parallel u\parallel_{L_t^6L_x^\frac{6n}{3n-8}(I\times\mathbb{R}^n)}^{\gamma-2}\parallel|\nabla|^ku\parallel_{U(I)}
.\end{eqnarray*}
 \end{lemma}
 \begin{proof}
Note that
\begin{eqnarray*}
\parallel |u|^pu\parallel_{\dot{N}^k(I\times\mathbb{R}^n)}&\lesssim&\parallel |\nabla|^k(|u|^pu)\parallel_{L_{t,x}^\frac{2(n+2)}{n+4}(I\times\mathbb{R}^n)},\\
\parallel(|x|^{-\gamma}\ast|u|^2)u\parallel_{\dot{N}^k(I\times\mathbb{R}^n)}&\lesssim&\parallel|\nabla|^k\left((|x|^{-\gamma}\ast|u|^2)u\right)\parallel_{L_t^2L_x^\frac{2n}{n+2}(I\times\mathbb{R}^n)}.
\end{eqnarray*}
Then using H\"older inequality and interpolation, one can get the
results.
 \end{proof}
 \section{Local Theory}
\ \quad Let's show the local theory for the initial value problem
\eqref{1}.  As the results are classical, we prefer to omit the
proofs and refer to \cite{16,19,20,22,23}.
\begin{proposition}\label{pro3.1}(Local well-posedness for \eqref{1} with
$H_x^1$-subcritical nonlinearities)\par Let $u_0\in H_x^1,
\lambda_1$ and $\lambda_2$ be nonzero real constants, with
$0<p<\frac{4}{n-2}, 0<\gamma<\min{\{n,4\}}$. Then, there exists
$T=T(\parallel u\parallel_{H_x^1})$ such that \eqref{1} with above
parameters admits a unique strong $H_x^1$-solution u on $[-T,T]$.
Let $(-T_{\mbox{min}}, T_{\mbox{max}})$ be the maximal time interval
on which the solution u is well-defined. For every compact time
interval $I\subset(-T_{\mbox{min}}, T_{\mbox{max}})$, we have $u\in
S^1(I\times\mathbb{R}^n)$ and the following properties hold:
\begin{itemize}
  \item If $T_{\mbox{max}}<\infty$(respectively,
  if $T_{\mbox{min}}<\infty$), then
  \[
  \parallel u(t)\parallel_{H_x^1}\rightarrow \infty\ \mbox{as}\ t\uparrow
  T_{\mbox{max}}\ (\mbox{respectively, as}\
  t\downarrow-T_{\mbox{min}}).
  \]
  \item The solution depends continuously on the initial value:\\
  There exists $T=T(\parallel u\parallel_{H_x^1})$ such that if $u_0^{(m)}\rightarrow
  u_0$ in $H_x^1$ and if $u^{(m)}$ is the solution to \eqref{1} with
  initial condition $u_0^{(m)}$, then $u^{(m)}$ is defined on $[-T,T]$
  for m sufficiently large and $u^{(m)}\rightarrow u$ in
  $S^1([-T,T]\times\mathbb{R}^n)$.
\end{itemize}
\end{proposition}
\begin{proposition}(Local well-posedness for \eqref{1} with a
$H_x^1$-critical nonlinearity)\par Let $u_0\in H_x^1, \lambda_1$ and
$\lambda_2$ be nonzero real constants.
\begin{itemize}
\item when $p=\frac{4}{n-2}$, and $0<\gamma<\min{\{n,4\}}$, for every
$T>0$, there exists $\eta=\eta(T)$ such that if
\[
\parallel e^{it\Delta}u_0\parallel_{\dot{X}^1([-T,T])}\leq\eta,
\]
then \eqref{1} with the parameters given above admits a unique
strong $H_x^1$-solution u defined $[-T,T]$;
\item when $0<p<\frac{4}{n-2},\ \gamma=4$ and $n\geq 5$, for every
$T>0$, there exists $\eta=\eta(T)$ such that if
\[
\parallel e^{it\Delta}u_0\parallel_{\dot{Y}^1([-T,T])}\leq\eta,
\]
then \eqref{1} with the parameters given above admits a unique
strong $H_x^1$-solution u defined on $[-T,T]$;
\item Let $(-T_{\mbox{min}},
T_{\mbox{max}})$ be the maximal time interval on which the solution
u is well-defined. Then $u\in S^1(I\times\mathbb{R}^n)$ for each
compact time interval $I\subset(-T_{\mbox{min}}, T_{\mbox{max}})$
and the following blow
up alternative hold:\\
If $T_{\mbox{max}}<\infty$(respectively,
  if $T_{\mbox{min}}<\infty$), then \\
  either $\parallel u(t)\parallel_{H_x^1}\rightarrow
  \infty$ or $\parallel u(t)\parallel_{S^1\left((0,t)\times\mathbb{R}^n\right)}\rightarrow
  \infty$ as $t\uparrow
  T_{\mbox{max}}$ (respectively, as $t\downarrow-T_{\mbox{min}}$).
\end{itemize}
\end{proposition}
\ \quad Next, we will establish the stability results for the
$H_x^1$-critical and the $L_x^2$-critical $\textsl{NLS}$ with
Hartree type.
\begin{lemma}(Short-time perturbation)\par
Let I be a compact interval, and let $\tilde{u}$ be a function on
$I\times\mathbb{R}^n$ which is a near-solution to \eqref{2} in the
sense that
\[
(i\partial_t+\Delta)\tilde{u}=\lambda(|x|^{-4}\ast|\tilde{u}|^2)\tilde{u}+e
\]for some function $e$.
Suppose that we have the energy bound
\begin{eqnarray}
\parallel \tilde{u}\parallel_{L_t^\infty
\dot{H}^1(I\times\mathbb{R}^n)}\leq E\end{eqnarray}for some $E>0$.

Let $t_0\in I$, and let $u(t_0)$ be close to $\tilde{u}(t_0)$ in the
sense that
\begin{eqnarray}
\parallel u(t_0)-\tilde{u}(t_0)\parallel_{\dot{H}_x^1}\leq E',
\end{eqnarray}
for some $E'>0$. Assume also that we have the smallness conditions
\begin{eqnarray}
\parallel
\nabla\tilde{u}\parallel_{U(I)}&\leq&\epsilon_0,\label{3.3}\\
\parallel e^{i(t-t_0)\Delta}\nabla
(u(t_0)-\tilde{u}(t_0))\parallel_{U(I)}&\leq&\epsilon,\label{3.4}\\
\parallel
e\parallel_{\dot{N}^1(I\times\mathbb{R}^n)}&\leq&\epsilon,\label{3.5}\end{eqnarray}for
some $0<\epsilon<\epsilon_0$, where $\epsilon_0$ is a small constant
$\epsilon_0=\epsilon_0(E,E')>0$.

We conclude that there exists a solution u to \eqref{2} on
$I\times\mathbb{R}^n$ with the special initial datum $u(t_0)$ at
$t_0$, and furthermore,
\begin{eqnarray}
\parallel
u-\tilde{u}\parallel_{\dot{S}^1(I\times\mathbb{R}^n)}\lesssim
E'+\epsilon,\\
\parallel
u\parallel_{\dot{S}^1(I\times\mathbb{R}^n)}\lesssim E'+E,\\
\parallel
u-\tilde{u}\parallel_{L_t^6L_x^{\frac{6n}{3n-8}}(I\times\mathbb{R}^n)}\lesssim
\epsilon,\\
\parallel(i\partial_t+\Delta)(
u-\tilde{u})\parallel_{\dot{N}^1(I\times\mathbb{R}^n)}\lesssim\epsilon.
\end{eqnarray}
\end{lemma}
\begin{proof}
Without loss of generality, we assume $t_0=\inf I$. Define
$z=u-\tilde{u}$, then $u=z+\tilde{u}$
\[
S(t):=\parallel(i\partial_t+\Delta)z\parallel_{\dot{N}^1([t_0,t]\times\mathbb{R}^n)}.
\]
By using H\"older, Hardy-Littlewood-Sobelov inequality, we have
\begin{eqnarray}\label{3.10}
\parallel\left(|x|^{-4}\ast(ab)\right)c\parallel_{\dot{N}^1}
&\lesssim&\parallel\nabla\left[\left(|x|^{-4}\ast(ab)\right)c\right]\parallel_{L_t^2L_x^{\frac{2n}{n+2}}}\nonumber\\
&\lesssim&
\parallel
\nabla a\parallel_{U(I)}\parallel \nabla b\parallel_{U(I)}\parallel
\nabla c\parallel_{U(I)},
\end{eqnarray}
and  from \eqref{3.3},\eqref{3.5} and \eqref{3.10}, we have
\begin{eqnarray}
S(t)&\leq&\parallel\left[|x|^{-4}\ast(|z|^2+z\bar{\tilde{u}}+\bar{z}\tilde{u})\right](z+\tilde{u})\parallel_{\dot{N}^1}+\parallel(|x|^{-4}\ast|\tilde{u}|^2)z\parallel_{\dot{N}^1}+\parallel
e\parallel_{\dot{N}^1}\nonumber\\
&\lesssim&\epsilon+\sum\limits_{j=0}^2\parallel\nabla
z\parallel_{U(I)}^j\parallel \nabla
\tilde{u}\parallel_{U(I)}^{3-j}\nonumber\\
&\lesssim&\epsilon+\sum\limits_{j=0}^2\epsilon_0^{3-j}\parallel\nabla
z\parallel_{U(I)}^j.\nonumber
\end{eqnarray}
On the other hand, one has
\begin{eqnarray}\label{3.11}
\parallel\nabla
z\parallel_{U(I)}\lesssim\parallel
e^{\left(i(t-t_0)\Delta\right)}\nabla
z(t_0)\parallel_{U(I)}+S(t)\lesssim S(t)+\epsilon,
\end{eqnarray}
and
\[
S(t)\lesssim\epsilon+\sum\limits_{j=0}^2\epsilon_0^{3-j}(S(t)+\epsilon)^j
\]
By a standard continunity method, one can show that
$S(t)\lesssim\epsilon$, then from \eqref{3.11} and Sobolev
embedding, we get
\begin{eqnarray*}
\parallel
u-\tilde{u}\parallel_{L_t^6L_x^{\frac{6n}{3n-8}}}&\lesssim&\epsilon\\
\parallel \tilde{u}\parallel_{\dot{S}^1}&\lesssim&\parallel
\tilde{u}(t_0)\parallel_{\dot{H}^1}+\parallel \nabla
\tilde{u}\parallel_{U(I)}^3+\parallel
e\parallel_{\dot{N}^1}\lesssim E+\epsilon_0^3+\epsilon\lesssim E\\
\parallel u-\tilde{u}\parallel_{\dot{S}^1}&\lesssim&\parallel
u(t_0)-\tilde{u}(t_0)\parallel_{\dot{H}_x^1}+S(t)\lesssim
E'+\epsilon.
\end{eqnarray*}
At last, we have
\[
\parallel
u\parallel_{\dot{S}^1}\lesssim\parallel
u-\tilde{u}\parallel_{\dot{S}^1}+\parallel
\tilde{u}\parallel_{\dot{S}^1}\lesssim E+E'.
\]
\end{proof}
\begin{remark}
If $\parallel
u(t_0)-\tilde{u}(t_0)\parallel_{\dot{H}_x^1}\leq\epsilon_0$, then,
thanks to the Strichartz estimate, we have
\[
\parallel e^{i(t-t_0)\Delta}\nabla
(u(t_0)-\tilde{u}(t_0))\parallel_{U(I)}\lesssim\parallel
u(t_0)-\tilde{u}(t_0)\parallel_{\dot{H}_x^1}\leq\epsilon_0.
\]
Therefore, if $E'$ is small, then \eqref{3.4} obviously holds true.
\end{remark}
\begin{lemma}\label{lemma3.2}($H_x^1$-critical stability result for Hartree type)\par
Let I be a compact interval, $t_0\in I$,  $\tilde{u}$ be a function
on $I\times\mathbb{R}^n$ which is a near-solution to \eqref{2} in
the sense that
\[
(i\partial_t+\Delta)\tilde{u}=\lambda(|x|^{-4}\ast|\tilde{u}|^2)\tilde{u}+e\qquad\mbox{for
some function e},
\]
and $u(t_0)$ be close to $\tilde{u}(t_0)$ in the sense that
\begin{eqnarray}
\parallel u(t_0)-\tilde{u}(t_0)\parallel_{\dot{H}_x^1}\leq E'\qquad\mbox{for some
}E'>0.
\end{eqnarray}
Suppose that we  have the energy bound
\begin{eqnarray}
\parallel \tilde{u}\parallel_{L_t^\infty
\dot{H}^1(I\times\mathbb{R}^n)}\leq E\qquad\mbox{for some }E>0,
\end{eqnarray}
and we also  have the following conditions
\begin{eqnarray}
\parallel
\nabla\tilde{u}\parallel_{U(I)}&\leq& M\qquad\mbox{for some }M>0,\\
\parallel e^{i(t-t_0)\Delta}\nabla
(u(t_0)-\tilde{u}(t_0))\parallel_{U(I)}&\leq&\epsilon,\label{3.15}\\
\parallel
e\parallel_{\dot{N}^1(I\times\mathbb{R}^n)}&\leq&\epsilon\end{eqnarray}
for some $0<\epsilon<\epsilon_0$, where $\epsilon_0$ is a small
constant $\epsilon_0=\epsilon_0(E,E',M)>0$.

Then, there exists a solution u to \eqref{2} on
$I\times\mathbb{R}^n$ with the special initial datum $u(t_0)$ at
$t_0$, satisfying
\begin{eqnarray}
\parallel
u-\tilde{u}\parallel_{\dot{S}^1(I\times\mathbb{R}^n)}&\lesssim&
C(M,E)(E'+\epsilon),\\
\parallel
u\parallel_{\dot{S}^1(I\times\mathbb{R}^n)}&\lesssim& C(M,E',E),\\
\parallel
u-\tilde{u}\parallel_{L_t^6L_x^{\frac{6n}{3n-8}}(I\times\mathbb{R}^n)}&\lesssim&
C(M,E,E')\epsilon.
\end{eqnarray}
\end{lemma}
\begin{proof}
Without loss of generality, we assume $t_0=\inf I$. Split $I$ into J
intervals $I_j$, such that on each $I_j$ we have
\[
\parallel
\nabla\tilde{u}\parallel_{U(I_j)}\leq\epsilon_0,\qquad\qquad\mbox{then
}J\sim\left(1+\frac{M}{\epsilon_0}\right)^6.
\]
Fix $I_0=[t_0,t_1]$, thanks to the short-time perturbation, one can
get
\begin{eqnarray*}
\parallel
u-\tilde{u}\parallel_{\dot{S}^1(I_0\times\mathbb{R}^n)}&\lesssim&
E'+\epsilon,\\
\parallel
u\parallel_{\dot{S}^1(I_0\times\mathbb{R}^n)}&\lesssim& E'+E,\\
\parallel
u-\tilde{u}\parallel_{L_t^6L_x^{\frac{6n}{3n-8}}(I_0\times\mathbb{R}^n)}&\lesssim&
\epsilon,\\
\parallel(i\partial_t+\Delta)(
u-\tilde{u})\parallel_{\dot{N}^1(I_0\times\mathbb{R}^n)}&\lesssim&\epsilon.
\end{eqnarray*}
Furthermore, we have\begin{equation*} \parallel
u(t_1)-\tilde{u}(t_1)\parallel_{\dot{H}_x^1}\le\parallel
u-\tilde{u}\parallel_{\dot{S}_x^1(I_0\times\mathbb{R}^n)}\lesssim
E'+\epsilon
\end{equation*}and
\begin{eqnarray*}
\parallel e^{i(t-t_1)\Delta}\nabla
(u(t_1)-\tilde{u}(t_1))\parallel_{U(I_1)}&\lesssim&\parallel
e^{i(t-t_0)\Delta}\nabla
(u(t_0)-\tilde{u}(t_0))\parallel_{U(I_1)}\\
+\parallel(i\partial_t+\Delta)(u-\tilde{u})\parallel_{N^1(I_0\times\mathbb{R}^n)}&\lesssim&\epsilon.
\end{eqnarray*}
Choosing $\epsilon$ small enough, from the short-time perturbation,
we have the results also hold on $I_1$, continuing the inductive
argument, we get the above results at last.
\end{proof}
\begin{remark}
In our lemmas, the condition \eqref{3.15} is weeker than the
condition of what stated in \cite{3}, where they require that
\begin{eqnarray*}
&&\left(\sum\limits_N\parallel P_N\nabla
e^{\left(i(t-t_0)\Delta\right)}(u(t_0)-\tilde{u}(t_0))\parallel_{U(I)}^2\right)^\frac{1}{2}\\
&&+\left(\sum\limits_N\parallel P_N\nabla
e^{\left(i(t-t_0)\Delta\right)}(u(t_0)-\tilde{u}(t_0))\parallel_{L_t^3L_x^{\frac{6n}{3n-4}}(I\times\mathbb{R}^n)}^2\right)^\frac{1}{2}
\leq\epsilon
\end{eqnarray*}
In fact, for Hartree type the nonlinearity and derivatives of the
nonlinearity are Lipschitz continuity.
\end{remark}
\ \quad The same method can be used to prove the perturbation theory
of the $L_x^2$-critical $\textsl{NLS}$ with Hartree type. Note that,
by H\"older, Hardy-Littlewood-Sobolev inequality, we have
\begin{eqnarray}\label{3.20}
\parallel\left(|x|^{-2}\ast(ab)\right)c\parallel_{\dot{N}^0}
&\lesssim&\parallel\left(|x|^{-2}\ast(ab)\right)c\parallel_{L_t^2L_x^{\frac{2n}{n+2}}}\nonumber\\
&\lesssim&
\parallel
a\parallel_{U(I)}\parallel b\parallel_{U(I)}\parallel
c\parallel_{U(I)},
\end{eqnarray}
\eqref{3.20} instead of \eqref{3.10}, by using a similar argument as
above,   we can get the following result:
\begin{lemma}\label{lemma3.3}($L_x^2$-critical stability result for Hartree type)\par
Let $I$ be a compact interval, $t_0\in I$,  $\tilde{u}$ be a
function on $I\times\mathbb{R}^n$ which is a near-solution to
\eqref{3} in the sense that
\[
(i\partial_t+\Delta)\tilde{u}=\lambda(|x|^{-2}\ast|\tilde{u}|^2)\tilde{u}+e\qquad\mbox{for
some function e},
\] and $u(t_0)$ be close to $\tilde{u}(t_0)$ in the sense that
\begin{eqnarray}
\parallel u(t_0)-\tilde{u}(t_0)\parallel_{L_x^2(\mathbb{R}^n)}\leq M'\qquad\mbox{for some
}M'>0.
\end{eqnarray}
Suppose that we  have the mass bound
\begin{eqnarray}
\parallel \tilde{u}\parallel_{L_t^\infty
L_x^2(I\times\mathbb{R}^n)}\leq M\qquad\mbox{for some M}>0
\end{eqnarray}
and the following conditions hold true
\begin{eqnarray}
\parallel
\tilde{u}\parallel_{U(I)}&\leq& L\qquad\mbox{for some }L>0\\
\parallel e^{i(t-t_0)\Delta}
(u(t_0)-\tilde{u}(t_0))\parallel_{U(I)}&\leq&\epsilon\label{3.24}\\
\parallel
e\parallel_{\dot{N}^0(I\times\mathbb{R}^n)}&\leq&\epsilon\end{eqnarray}
for some $0<\epsilon<\epsilon_1$, where $\epsilon_1$ is a small
constant, $\epsilon_1=\epsilon_1(M,M',L)>0$.

Then, there exists a solution u to \eqref{3} on
$I\times\mathbb{R}^n$ with the special initial datum $u(t_0)$ at
$t_0$, and
\begin{eqnarray}
\parallel
u-\tilde{u}\parallel_{\dot{S}^0(I\times\mathbb{R}^n)}&\lesssim&
C(L,M,M')(M'+\epsilon),\\
\parallel
u\parallel_{\dot{S}^0(I\times\mathbb{R}^n)}&\lesssim& C(L,M,M'),\\
\parallel
u-\tilde{u}\parallel_{U(I)}&\lesssim& C(L,M,M')\epsilon.
\end{eqnarray}

\end{lemma}
\ \quad The corresponding stability results for the $H_x^1$-critical
and the $L_x^2$-critical $\textsl{NLS}$ with power type have been
established by \cite{24,25}. However, when the dimension n is
greater than 6, the case is more delicate as derivatives of the
nonlinearity are merely H\"older continuous of order $\frac{4}{n-2}$
rather than Lipshchitz. One can find the details in \cite{24,25}, we
state their result below:
\begin{lemma}\label{lemma3.4}($H_x^1$-critical stability result for power type)
Let I be a compact interval, $t_0\in I$, $\tilde{u}$ be a function
on $I\times\mathbb{R}^n$ which is a near-solution to \eqref{4} in
the sense that
\[
(i\partial_t+\Delta)\tilde{u}=\lambda|\tilde{u}|^{\frac{4}{n-2}}\tilde{u}+e\qquad\mbox{for
some function e},
\]and  $u(t_0)$ be close to $\tilde{u}(t_0)$ in the sense that
\begin{eqnarray}\label{3.30}
\parallel u(t_0)-\tilde{u}(t_0)\parallel_{\dot{H}_x^1}\leq E_0'\qquad\mbox{for some
}E_0'>0.
\end{eqnarray}
Suppose that we  have the energy bound
\begin{eqnarray}\label{3.29}
\parallel \tilde{u}\parallel_{L_t^\infty
\dot{H}^1(I\times\mathbb{R}^n)}\leq E_0\qquad\mbox{for some }E_0>0
\end{eqnarray}
and the following conditions to be true
\begin{eqnarray}
&&\hspace{-9cm}\parallel
\tilde{u}\parallel_{W(I)}\leq M_0\hspace{3.5cm}\mbox{for some }M_0>0\label{3.31}\\
&&\hspace{-9cm}\left(\sum\limits_N\parallel P_N\nabla
e^{\left(i(t-t_0)\Delta\right)}(u(t_0)-\tilde{u}(t_0))\parallel_{L_t^{\frac{2(n+2)}{n-2}}L_x^{\frac{2n(n+2)}{n^2+4}}(I\times\mathbb{R}^n)}^2\right)^\frac{1}{2}\leq\epsilon\label{3.32}\\
&&\hspace{-9cm}\parallel
e\parallel_{\dot{N}^1(I\times\mathbb{R}^n)}\leq\epsilon\\
\hspace{-2.5cm}\mbox{for some }0<\epsilon<\epsilon_2\mbox{, where
}\epsilon_2=\epsilon_2(E_0,E_0',M_0)\mbox{ is a small constant
}\nonumber
\end{eqnarray}
Then, there exists a solution $u$ to \eqref{4} on
$I\times\mathbb{R}^n$ with the special initial datum $u(t_0)$ at
$t_0$, and
\begin{eqnarray}
\parallel
u-\tilde{u}\parallel_{\dot{S}^1(I\times\mathbb{R}^n)}&\lesssim&
C(E_0,E'_0,M_0)(E_0'+\epsilon+\epsilon^{\frac{7}{(n-2)^2}}),\\
\parallel
u\parallel_{\dot{S}^1(I\times\mathbb{R}^n)}&\lesssim&C(M_0,E_0',E_0),\\
\parallel
u-\tilde{u}\parallel_{L_t^{\frac{2(n+2)}{n-2}}L_x^{\frac{2n(n+2)}{n^2+4}}(I\times\mathbb{R}^n)}&\lesssim&
C(M_0,E_0,E_0')(\epsilon+\epsilon^{\frac{7}{(n-2)^2}}).
\end{eqnarray}
\begin{remark}
From \cite{24} by Strichartz and Plancherel, on the slab
$I\times\mathbb{R}^n$ we have
\begin{eqnarray*}
\left(\sum\limits_N\parallel P_N\nabla
e^{\left(i(t-t_0)\Delta\right)}(u(t_0)-\tilde{u}(t_0))\parallel_{L_t^{\frac{2(n+2)}{n-2}}L_x^{\frac{2n(n+2)}{n^2+4}}(I\times\mathbb{R}^n)}^2\right)^\frac{1}{2}\\
&&\hspace{-3cm}\lesssim\left(\sum\limits_N\parallel P_N
\nabla(u(t_0)-\tilde{u}(t_0))\parallel^2_{L_t^\infty L_x^2}\right)^\frac{1}{2}\\
&&\hspace{-3cm}\lesssim\ \parallel\nabla(u(t_0)-\tilde{u}(t_0))\parallel^2_{L_t^\infty L_x^2}\\
&&\hspace{-3cm}\lesssim E_0'\\
\end{eqnarray*}
so the hypothesis \eqref{3.32} is redundant if $E_0'$ is small.
\end{remark}
\end{lemma}
\begin{lemma}\label{lemma3.5}($L_x^2$-critical stability result for power type)\par
Let $I$ be a compact interval, $t_0\in I$, $\tilde{u}$ be a function
on $I\times\mathbb{R}^n$ which is a near-solution to \eqref{5} in
the sense that
\[
(i\partial_t+\Delta)\tilde{u}=\lambda|\tilde{u}|^{\frac{4}{n}}\tilde{u}+e\qquad\mbox{for
some function e},
\]and  $u(t_0)$ be close to $\tilde{u}(t_0)$ in the sense
that
\begin{eqnarray}\label{3.38}
\parallel u(t_0)-\tilde{u}(t_0)\parallel_{L_x^2(\mathbb{R}^n)}\leq M_0'\qquad\mbox{for some
}M_0'>0.
\end{eqnarray}
Suppose that we  have the mass bound
\begin{eqnarray}
\parallel \tilde{u}\parallel_{L_t^\infty
L_x^2(I\times\mathbb{R}^n)}\leq M_0\qquad\mbox{for some }M_0>0,
\end{eqnarray}
and  the following conditions to be true
\begin{eqnarray}
\parallel
\tilde{u}\parallel_{V(I)}&\leq& L_0\qquad\mbox{for some }L_0>0,\label{3.39}\\
\parallel e^{i(t-t_0)\Delta}
(u(t_0)-\tilde{u}(t_0))\parallel_{V(I)}&\leq&\epsilon,\\
\parallel
e\parallel_{\dot{N}^0(I\times\mathbb{R}^n)}&\leq&\epsilon\end{eqnarray}
for some $0<\epsilon<\epsilon_3$, where $\epsilon_3$ is a small
constant $\epsilon_3=\epsilon_3(M_0,M_0',L_0)>0$.

Then, there exists a solution $u$ to \eqref{5} on
$I\times\mathbb{R}^n$ with the special initial datum $u(t_0)$ at
$t_0$, and furthermore,
\begin{eqnarray}
\parallel
u-\tilde{u}\parallel_{\dot{S}^0(I\times\mathbb{R}^n)}&\lesssim&
C(L_0,M_0,M_0')M_0',\\
\parallel
u\parallel_{\dot{S}^0(I\times\mathbb{R}^n)}&\lesssim&C(L_0,M_0,M_0'),\\
\parallel
u-\tilde{u}\parallel_{V(I)}&\lesssim&C(L_0,M_0,M_0')\epsilon.
\end{eqnarray}
\end{lemma}
\par
 To conclude this  section, we state the results involving
 persistence of $L^2$ or $\dot{H}^1$ regularity for critical $\textsl{NLS}$ with
 Hartree type or power type:
\begin{lemma}\label{lemma 3.6}(Persistence of regularity):
Let $k=0,1$, and $I$ be a compact interval, $t_0\in I$ .\\
case1:  $u$ is a solution to \eqref{2} on $I\times\mathbb{R}^n$
obeying the bounds
\[
\parallel
u\parallel_{L_t^6L_x^{\frac{6n}{3n-8}}(I\times\mathbb{R}^n)}\leq M.
\]
Then, if  $u(t_0)\in\dot{H}{_x^k}$, we have
\[
\parallel u\parallel_{\dot{S}^k(I\times\mathbb{R}^n)}\leq
C(M)\parallel u(t_0)\parallel_{\dot{H}{_x^k}}
\]
case 2: $u$ is a solution to \eqref{3} on $I\times\mathbb{R}^n$
obeying the bounds
\[
\parallel
u\parallel_{U(I)}\leq L
\]
Then, if $u(t_0)\in\dot{H}{_x^k}$, we have
\[
\parallel u\parallel_{\dot{S}^k(I\times\mathbb{R}^n)}\leq
C(L)\parallel u(t_0)\parallel_{\dot{H}{_x^k}}
\]
case 3:  $u$ is a solution to \eqref{4} on $I\times\mathbb{R}^n$
obeying the bounds
\[
\parallel
u\parallel_{W(I)}\leq M.
\]
Then, if  $u(t_0)\in\dot{H}{_x^k}$, we have
\[
\parallel u\parallel_{\dot{S}^k(I\times\mathbb{R}^n)}\leq
C(M)\parallel u(t_0)\parallel_{\dot{H}{_x^k}}.
\]
case 4: Let $u$ be a solution to \eqref{5} on $I\times\mathbb{R}^n$
obeying the bounds
\[
\parallel
u\parallel_{V(I)}\leq L.
\]
Then, if $u(t_0)\in\dot{H}{_x^k}$, we have
\[
\parallel u\parallel_{\dot{S}^k(I\times\mathbb{R}^n)}\leq
C(L)\parallel u(t_0)\parallel_{\dot{H}{_x^k}}.
\]
\end{lemma}
\begin{proof}
The method to prove these four cases is similar, we only consider
the first case, and the others are omitted.\par
 Subdivide the interval $I$ into $N\sim(1+\frac{M}{\eta})^6$
 subintervals $I_j=[t_j,t_{j+1}]$ such that
 \[
\parallel
u\parallel_{L_t^6L_x^{\frac{6n}{3n-8}}(I_j\times\mathbb{R}^n)}\leq\eta
 \]
where $\eta$ is a small positive constant to be chosen later. By
using Strichartz estimates, on each $I_j$ we obtain
\begin{eqnarray*}
\parallel u\parallel_{\dot{S}^k(I_j\times\mathbb{R}^n)}&\lesssim&\parallel
u(t_j)\parallel_{\dot{H}{_x^k}}+\parallel
u\parallel_{\dot{S}^k(I_j\times\mathbb{R}^n)}\parallel
u\parallel^2_{L_t^6L_x^{\frac{6n}{3n-8}}(I_j\times\mathbb{R}^n)}\\
&\lesssim&\parallel u(t_j)\parallel_{\dot{H}{_x^k}}+\eta^2\parallel
u\parallel_{\dot{S}^k(I_j\times\mathbb{R}^n)}
\end{eqnarray*}
Choosing $\eta$ sufficiently small, we get
\[
\parallel
u\parallel_{\dot{S}^k(I_j\times\mathbb{R}^n)}\lesssim\parallel
u(t_j)\parallel_{\dot{H}{_x^k}}.
\]
Next, we consider the relationship between $\parallel
u(t_j)\parallel_{\dot{H}{_x^k}}$ and $\parallel
u(t_0)\parallel_{\dot{H}{_x^k}}$. For $I_0$, we have
\[\parallel
u(t_1)\parallel_{\dot{H}{_x^k}}\leq\parallel
u\parallel_{\dot{S}^k(I_0\times\mathbb{R}^n)}\leq C\parallel
u(t_0)\parallel_{\dot{H}{_x^k}}.
\]
For $I_1$, we have
\[\parallel
u(t_2)\parallel_{\dot{H}{_x^k}}\leq\parallel
u\parallel_{\dot{S}^k(I_1\times\mathbb{R}^n)}\leq C\parallel
u(t_1)\parallel_{\dot{H}{_x^k}}\leq C^2\parallel
u(t_0)\parallel_{\dot{H}{_x^k}}
\]
by using iteration arguments, for each $I_j$ we can obtain:
\[
\parallel
u(t_j)\parallel_{\dot{H}{_x^k}}\leq C^j\parallel
u(t_0)\parallel_{\dot{H}{_x^k}}
\]
Adding these estimates over all the subinterval $I_j$, we can get
the results.
\end{proof}

\section{Global well-posedness}
\ \quad The aim of this section is to prove the Theorem \ref{1.1}.
For the convenience, we shall abbreviate the energy $E(u)$ to E, and
the mass $M(u)$ to $M$. In order to prove the global well-posedness
of \eqref{1}, we should state that the blowup couldn't hold. For
\eqref{1} with $\dot{H}_x^1-$subcritical nonlinearities, we should
prove that $\parallel u(t)\parallel_{H_x^1}$ is bounded for all time
where the solution is defined. We notice the conservation of mass,
so we focus  on the bounds of $\parallel
u(t)\parallel_{\dot{H}_x^1}$. For \eqref{1} with a
$\dot{H}_x^1-$critical nonlinearity, we view the energy-subcritical
nonlinearity as a perturbation to the energy-critical
$\textsl{NLS}$, which is globally well-posed. For any compact
interval $I$, $u$ is the strong solution to \eqref{1} which is
defined on $I\times\mathbb{R}^n$. $u_0\in H_x^1$ is the initial
datum.
\subsection{Bound State}
\ \quad Let $R(x)$ and $W(x)$ be the positive radial Schwartz
solution of the ground state to the elliptic equations respectively:
\begin{eqnarray*}
\Delta R+|R|^pR&=&\frac{4-(n-2)p}{np}R,\\
\Delta
W+\left(|x|^{-\gamma}\ast|w|^2\right)W&=&\frac{4-\gamma}{\gamma}W.
\end{eqnarray*}
From the work of \cite{1,8,19} and \cite{26}, we have the following
characterization of $R$ and $W$:
\begin{eqnarray}
 \parallel u\parallel_{L^{p+2}}^{p+2}&\leq& C_R\parallel
\nabla u\parallel_{L^2}^{\frac{np}{2}}\parallel
u\parallel_{L^2}^{\frac{4-(n-2)p}{2}},\qquad\forall\ u,v \in H_x^1,\label{4.1}\\
\parallel
\left(|x|^{-\gamma}\ast|v|^2\right)|v|^2\parallel_{L^1}&\leq&
C_W\parallel \nabla v\parallel_{L^2}^\gamma\parallel
v\parallel_{L^2}^{4-\gamma},\label{4.2}
\end{eqnarray}
where $C_R$ and $C_W$ is the best constant for their respective
inequality, moreover
 \begin{eqnarray*}C_R&=&\frac{2(p+2)}{np}\parallel \nabla
R\parallel_{L^2}^{-p}=\frac{2(p+2)}{np}\parallel
R\parallel_{L^2}^{-p},\\
C_W&=&\frac{4}{\gamma}\parallel \nabla
W\parallel_{L^2}^{-2}=\frac{4}{\gamma}\parallel
W\parallel_{L^2}^{-2}.
\end{eqnarray*}
If we define\begin{eqnarray*} \tilde{E}(R):&=&\frac{1}{2}\int|\nabla
R|^2\,dx-\frac{1}{p+2}\int|R|^{p+2}\,dx,\\
\tilde{E}(W):&=&\frac{1}{2}\int|\nabla
W|^2\,dx-\frac{1}{4}\int\left(|x|^{-\gamma}\ast|w|^2\right)|W|^2\,dx,
\end{eqnarray*}
then, we have
\begin{eqnarray*}
\tilde{E}(R)&=&\left(\frac{1}{2}-\frac{2}{np}\right)\int|\nabla
R|^2\,dx=\left(\frac{1}{2}-\frac{2}{np}\right)\left(\frac{2(p+2)}{npC_R}\right)^{\frac{2}{p}},\\
\tilde{E}(W)&=&\left(\frac{1}{2}-\frac{1}{\gamma}\right)\int|\nabla
W|^2\,dx=\frac{2(\gamma-2)}{\gamma^2C_W}.
\end{eqnarray*}

Define $E_1:=\frac{1}{2}\int|\nabla
u|^2\,dx-\frac{|\lambda_1|}{p+2}\int|u|^{p+2}\,dx$, where
$\lambda_1$ is the constant in \eqref{1}.
\begin{lemma}
Assume that
\begin{eqnarray*}
\parallel \nabla u\parallel_{L^2}^{2}\left(\parallel
u\parallel_{L^2}^2\right)^\frac{4-(n-2)p}{np-4}&<&|\lambda_1|^{\frac{4}{4-np}}\parallel
\nabla R\parallel_{L^2}^{\frac{4p}{np-4}},\\
E_1\cdot\left(\parallel
u\parallel_{L^2}^2\right)^\frac{4-(n-2)p}{np-4}&\leq&(1-\delta_0)|\lambda_1|^{\frac{4}{4-np}}\left(\frac{2np}{np-4}\right)^{\frac{4-(n-2)p}{np-4}}
\left(\tilde{E}(R)\right)^\frac{2p}{np-4},\quad\mbox{where
}\delta_0>0.
\end{eqnarray*}
Then, when $\frac{4}{n}<p\leq\frac{4}{n-2}$, there exists a
$\bar{\delta}=\bar{\delta}(\delta_0,n)>0$ such that
\begin{eqnarray*}
&&\parallel \nabla u\parallel_{L^2}^{2}\left(\parallel
u\parallel_{L^2}^2\right)^\frac{4-(n-2)p}{np-4}\leq(1-\bar{\delta})|\lambda_1|^{\frac{4}{4-np}}\parallel
\nabla R\parallel_{L^2}^{\frac{4p}{np-4}},\\
&&\hspace{-2cm}\mbox{and}\qquad\qquad E_1\geq 0.
\end{eqnarray*}
\end{lemma}
\begin{proof}
By $\eqref{4.1}$, we get $$E_1\geq\frac{1}{2}\int|\nabla
u|^2\,dx-\frac{|\lambda_1|}{p+2}C_R\parallel \nabla
u\parallel_{L^2}^{\frac{np}{2}}\parallel
u\parallel_{L^2}^{\frac{4-(n-2)p}{2}}.$$ Let
$$f(x)=\frac{1}{2}x-\frac{|\lambda_1|}{p+2}C_R\parallel
u\parallel_{L^2}^{\frac{4-(n-2)p}{2}}x^{\frac{np}{4}}$$ and
$a=\int|\nabla u|^2\,dx$. Note that $$\hspace{1cm}f'(x)=0\
\Leftrightarrow\ x=|\lambda_1|^{\frac{4}{4-np}}\parallel
u\parallel_{L^2}^{-\frac{2[4-(n-2)p]}{np-4}}\parallel \nabla
R\parallel_{L^2}^{\frac{4p}{np-4}}:=x_0,$$ and,
 \begin{eqnarray*}
&&f'(x)>0\quad \mbox{for\ }x<x_0,\\
&&f(0)=0,f(x_0)=(\frac{1}{2}-\frac{2}{np})|\lambda_1|^{\frac{4}{4-np}}\left(\frac{2np}{np-4}\right)^{\frac{4-(n-2)p}{np-4}}
\left(\parallel
u\parallel_{L^2}^2\right)^{-\frac{4-(n-2)p}{np-4}}\left(\tilde{E}(R)\right)^\frac{2p}{np-4},
 \end{eqnarray*}
 using the fact that
 $a\in[0,x_0)$, and the condition $E_1\leq(1-\delta_0)f(x_0)$, we can get that there
 exists
 $\bar{\delta}=\bar{\delta}(\delta_0,n)$ such that
 \begin{eqnarray*}
a\leq(1-\bar{\delta})x_0\qquad\mbox{and}\qquad E_1\geq f(a)\geq0.
\end{eqnarray*}
\end{proof}
\par
Let's define $E_2:=\frac{1}{2}\int|\nabla
v|^2\,dx-\frac{|\lambda_2|}{4}\int\left(|x|^{-\gamma}\ast|v|^2\right)|v|^2\,dx$,
where $\lambda_2$ is the constant in \eqref{1}. The same result can
be gotten for $W(x)$:
\begin{lemma}\label{lemma4.2}
Assume that
\begin{eqnarray*}
\parallel \nabla v\parallel_{L^2}^{2}\left(\parallel
v\parallel_{L^2}^2\right)^\frac{4-\gamma}{\gamma-2}&<&\left(\frac{\parallel\nabla W\parallel_{L^2}^2}{|\lambda_2|}\right)^{\frac{2}{\gamma-2}},\\
E_2\cdot\left(\parallel
v\parallel_{L^2}^2\right)^\frac{4-\gamma}{\gamma-2}&\leq&(1-\delta_0)(\frac{1}{2}-\frac{1}{\gamma})
\left[\frac{2\gamma\tilde{E}(W)}{|\lambda_2|(\gamma-2)}\right]^{\frac{2}{\gamma-2}},
\end{eqnarray*}where $\delta_0>0$.
Then, when $2<\gamma\leq4$, there exists a
$\bar{\delta}=\bar{\delta}(\delta_0,n)>0$ such that
\begin{eqnarray*}
&&\parallel \nabla v\parallel_{L^2}^{2}\left(\parallel
v\parallel_{L^2}^2\right)^\frac{4-\gamma}{\gamma-2}\leq(1-\bar{\delta})
\left(\frac{\parallel\nabla W\parallel_{L^2}^2}{|\lambda_2|}\right)^{\frac{2}{\gamma-2}}\\
&&\hspace{-2cm}\mbox{and}\qquad\qquad E_2\geq0.
\end{eqnarray*}
\end{lemma}
\subsection{Kinetic energy control}
\ \quad We'll get a prior control on the kinetic energy, which is
bounded for all time for which the solution is defined. More
precisely, the bound is only concerned with energy and mass, i.e.
\begin{eqnarray}
\parallel u(t)\parallel_{\dot{H}_x^1}\leq C(E,M).
\end{eqnarray}
We observe the energy
\[
E(u)=\frac{1}{2}\int|\nabla
u|^2\,dx+\frac{\lambda_1}{p+2}\int|u|^{p+2}\,dx+\frac{\lambda_2}{4}\int\left(|x|^{-\gamma}\ast|u|^2\right)|u|^2\,dx
\]
is conserved. Hence,\\
for the  case $\eqref{case 1}$, we obviously obtain
$\hspace{2cm}\parallel u(t)\parallel_{\dot{H}_x^1}\lesssim E.$\\
For the case $\eqref{case 2}$, from Parseval identity,
Hardy-Littlewood-Sobolev inequality and interpolation, we have
\begin{eqnarray}
\int\left(|x|^{-\gamma}\ast|u|^2\right)|u|^2\,dx:&=&\int\left(|\nabla|^{-(n-\gamma)}|u|^2\right)|u|^2\,dx
=\parallel|\nabla|^{-\frac{n-\gamma}{2}}|u|^2\parallel_{L^2}^2\nonumber\\
&\leq&\parallel u\parallel_{L^{\frac{4n}{2n-\gamma}}}^4\leq\parallel
u\parallel_{L^2}^{4(1-\frac{n+\gamma}{2n})}\parallel
u\parallel_{L^\frac{2n+2\gamma}{n}}^{\frac{2n+2\gamma}{n}}.
\end{eqnarray}
Based on the fact: {\it for any positive constants $a$, $\delta$,
and $p_1<p_2$, the following inequality
\begin{eqnarray}
a^{p_1+2}\leq C(\delta)a^2+\delta a^{p_2+2}\label{4.5}
\end{eqnarray} holds true,}
 we can get
\[
\parallel
u\parallel_{L^\frac{2n+2\gamma}{n}}^{\frac{2n+2\gamma}{n}}\leq
C(\delta)\parallel u\parallel_{L^2}^2+\delta\parallel
u\parallel_{L^{p+2}}^{p+2}
\] if $\frac{2n+2\gamma}{n}<p+2$, i.e. $\gamma<\frac{np}{2}$.
Then $$E(u)\geq\frac{1}{2}\int|\nabla
u|^2\,dx+\frac{\lambda_1}{p+2}\int|u|^{p+2}\,dx-C\parallel
u\parallel_{L^2}^{4(1-\frac{n+\gamma}{2n})}\delta\int|u|^{p+2}\,dx-C(M).$$
Let $\delta$ is small enough, we have $$\parallel
u(t)\parallel_{\dot{H}_x^1}\leq C(E,M).$$ If
$\gamma\geq\frac{np}{2}$, by using $\lambda_1>0$ and $\eqref{4.2}$,
we can obtain $$E\geq E_1\geq\frac{1}{2}\int|\nabla
u|^2\,dx-\frac{|\lambda_2|}{4}C_W\parallel \nabla
u\parallel_{L^2}^\gamma\parallel u\parallel_{L^2}^{4-\gamma}.$$ For
the case $\gamma<2$, from Young's inequality, one has
\[
E\geq\frac{1}{2}\int|\nabla u|^2\,dx-\frac{|\lambda_2|}{4}\delta
C_W\parallel \nabla
u\parallel_{L^2}^2-\frac{|\lambda_2|}{4}C_WC(\delta)\parallel
u\parallel_{L^2}^{\frac{2(4-\gamma)}{2-\gamma}}.
\]
Let $\delta$ be small enough, we obtain $$\hspace{1cm}\parallel
u(t)\parallel_{\dot{H}_x^1}\leq C(E,M).$$ When $\gamma=2$, we have
$$\ E\geq\left(\frac{1}{2}-\frac{|\lambda_2|}{4}C_W\parallel
u\parallel_{L^2}^2\right)\parallel \nabla u\parallel_{L^2}^2.$$ If
$$\parallel
u\parallel_{L^2}^2<\frac{2}{C_W|\lambda_2|}=\frac{1}{|\lambda_2|}\parallel
W\parallel_{L^2}^2$$ holds true, we can obtain $$\parallel
u(t)\parallel_{\dot{H}_x^1}\leq C(E,M).$$ For the case $2<\gamma\leq
4$, by using Lemma $\ref{lemma4.2}$ and the conservation of energy
and mass, we only need to show when $$\parallel \nabla
u_0\parallel_{L^2}^{2}\left(\parallel
u_0\parallel_{L^2}^2\right)^\frac{4-\gamma}{\gamma-2}<\left(\frac{\parallel\nabla
W\parallel_{L^2}^2}{|\lambda_2|}\right)^{\frac{2}{\gamma-2}},$$ we
can get $$\parallel \nabla u\parallel_{L^2}^{2}\left(\parallel
u\parallel_{L^2}^2\right)^\frac{4-\gamma}{\gamma-2}<\left(\frac{\parallel\nabla
W\parallel_{L^2}^2}{|\lambda_2|}\right)^{\frac{2}{\gamma-2}}.$$ We
prove it by the continuity argument. Define
\begin{eqnarray*}
\Omega&=&\left\{t\in I,\parallel \nabla
u\parallel_{L^2}^{2}\left(\parallel
u\parallel_{L^2}^2\right)^\frac{4-\gamma}{\gamma-2}<\left(\frac{\parallel\nabla
W\parallel_{L^2}^2}{|\lambda_2|}\right)^{\frac{2}{\gamma-2}},\right.\\
&& \left.E\left(\parallel
u\parallel_{L^2}^2\right)^\frac{4-\gamma}{\gamma-2}\leq(1-\delta_0)(\frac{1}{2}-\frac{1}{\gamma})
\left[\frac{2\gamma\tilde{E}(W)}{|\lambda_2|(\gamma-2)}\right]^{\frac{2}{\gamma-2}}
\right\}.
\end{eqnarray*}
It suffices to show $\Omega$ is both open and closed. Note that
$t_0\in \Omega$, the open of $\Omega$ is obvious because of $u\in
C_t^0(I,\dot{H}_x^1)$. Therefore, we only need to prove $\Omega$ is
closed. For any $t_n\in\Omega,\ T\in I$, such that $t_n\rightarrow
T$, we have
\begin{eqnarray*}
\parallel \nabla
u(t_n)\parallel_{L^2}^{2}M^{\frac{4-\gamma}{\gamma-2}}&<&\left(\frac{\parallel\nabla
W\parallel_{L^2}^2}{|\lambda_2|}\right)^{\frac{2}{\gamma-2}},\\
E\left(u(t_n)\right)M^{\frac{4-\gamma}{\gamma-2}}&\leq&(1-\delta_0)(\frac{1}{2}-\frac{1}{\gamma})\left[\frac{2\gamma\tilde{E}(W)}{|\lambda_2|(\gamma-2)}\right]^{\frac{2}{\gamma-2}}
.\end{eqnarray*} By using Lemma $\ref{lemma4.2}$, we can get
\[
\parallel \nabla
u(t_n)\parallel_{L^2}^{2}M^{\frac{4-\gamma}{\gamma-2}}\leq(1-\bar{\delta})\left(\frac{\parallel\nabla
W\parallel_{L^2}^2}{|\lambda_2|}\right)^{\frac{2}{\gamma-2}}.
\]
Since $u\in C_t^0(I,\dot{H}_x^1)$, the conservation of energy and
mass, we get
\begin{eqnarray*}
\parallel \nabla
u(T)\parallel_{L^2}^{2}M^{\frac{4-\gamma}{\gamma-2}}&\leq&(1-\bar{\delta})\left(\frac{\parallel\nabla
W\parallel_{L^2}^2}{|\lambda_2|}\right)^{\frac{2}{\gamma-2}},\\
E\left(u(T)\right)M^{\frac{4-\gamma}{\gamma-2}}&\leq&(1-\delta_0)(\frac{1}{2}-\frac{1}{\gamma})\left[\frac{2\gamma\tilde{E}(W)}{|\lambda_2|(\gamma-2)}\right]^{\frac{2}{\gamma-2}}.
\end{eqnarray*}
This implies that $T\in\Omega$ and  $\hspace{1cm}\parallel
u(t)\parallel_{\dot{H}_x^1}\leq C(E,M)$.\\
\begin{remark}
When $\gamma=\frac{np}{2}$, we have
\begin{eqnarray*}
E\geq\frac{1}{2}\int|\nabla
u|^2\,dx+\frac{|\lambda_1|}{p+2}\parallel
u\parallel_{L^{p+2}}^{p+2}-C\frac{|\lambda_2|}{4}M^\frac{2-p}{2}\parallel
u\parallel_{L^{p+2}}^{p+2}.
\end{eqnarray*}
The condition $n>\gamma=\frac{np}{2}$ implies  $p<2$. If requiring
$\frac{|\lambda_1|}{p+2}>C\frac{|\lambda_2|}{4}M^\frac{2-p}{2}$,
i.e.
$M<\left(\frac{4|\lambda_1|}{(p+2)C|\lambda_2|}\right)^{\frac{2}{2-p}}$,
we also can obtain $\hspace{1cm}\parallel
u(t)\parallel_{\dot{H}_x^1}\leq C(E,M)$.
\end{remark}
To prove the case $\ref{case 3}$,  we need the following lemma
before getting the prior control on the kinetic energy:
\begin{lemma}\label{lemma4.3}
\begin{equation}\label{4.6}
\parallel
|\nabla|^{-\frac{n-\gamma}{4}} f\parallel _{L^4}\lesssim
\parallel
|\nabla|^{-\frac{n-\gamma}{2}}|f|^2\parallel_{L^2}^{\frac{1}{2}}.
\end{equation}
\end{lemma}
\begin{remark}
T. Tao proved the inequality for $\gamma=3$ in \cite{24}. We can use
the same method to get \eqref{4.6}.
\end{remark}
\begin{proof}
It suffices to prove \eqref{4.6} for a positive Schwartz function
$f$. In fact, we only need to prove the pointwise inequality
\begin{equation}\label{4.7}
S(|\nabla|^{-\frac{n-\gamma}{4}}
f)(x)\lesssim\left[(|\nabla|^{-\frac{n-\gamma}{2}}|f|^2)(x)\right]^{\frac{1}{2}},
\end{equation}
where $Sf:=(\sum_N|P_Nf|^2)^\frac{1}{2}$. \\
Obviously, \eqref{4.7} implies \eqref{4.6}.
\[
\parallel
|\nabla|^{-\frac{n-\gamma}{4}} f\parallel _{L^4}\lesssim\parallel
S(|\nabla|^{-\frac{n-\gamma}{4}} f)\parallel _{L^4}\lesssim\parallel
(|\nabla|^{-\frac{n-\gamma}{2}}|f|^2)^{\frac{1}{2}}\parallel_{L^4}\lesssim\parallel
|\nabla|^{-\frac{n-\gamma}{2}}|f|^2\parallel_{L^2}^{\frac{1}{2}}.
\]
Subsequently, we'll focus our attention to the estimate for each of
the dyadic pieces
\[
P_N(|\nabla|^{-\frac{n-\gamma}{4}} f)(x)=\int e^{2\pi ix\cdot\,\xi}
\hat{f}(\xi)|\xi|^{-\frac{n-\gamma}{4}}m(\frac{\xi}{N})\,d\xi,
\]
where $m(\xi):=\varphi(\xi)-\varphi(2\xi)$ in the notation
introduced in Section 2.\\
As $|\xi|^{-\frac{n-\gamma}{4}}m(\frac{\xi}{N})\thicksim
N^{-\frac{n-\gamma}{4}}m(\frac{\xi}{N})$, we have
\[
P_N(|\nabla|^{-\frac{n-\gamma}{4}} f)(x)\thicksim
N^{\frac{3n+\gamma}{4}}f\ast
\check{m}(Nx)=N^{\frac{3n+\gamma}{4}}\int f(x-y)\check{m}(Ny)\,dy.
\]
Since $m$ is a Schwartz function, we have
\[
|P_N(|\nabla|^{-\frac{n-\gamma}{4}} f)(x)|\lesssim
N^{\frac{3n+\gamma}{4}}\int_{|y|\leq N^{-1}}
f(x-y)\,dy+N^{\frac{3n+\gamma}{4}}\int_{|y|>
N^{-1}}f(x-y)\frac{1}{|Ny|^\beta}\,dy,
\]
where $\beta$ is chosen later.\\
A simple application of Cauchy-Schwartz yields
\begin{eqnarray*}
S(|\nabla|^{-\frac{n-\gamma}{4}}
f)(x)&=&\left(\sum_N|P_N(|\nabla|^{-\frac{n-\gamma}{4}}
f)(x)|^2\right)^\frac{1}{2}\\
&\lesssim&\left(\sum_NN^{\frac{3n+\gamma}{2}}|\int_{|y|\leq N^{-1}}
f(x-y)\,dy|^2+\sum_NN^{\frac{3n+\gamma}{2}}|\int_{|y|>N^{-1}}
f(x-y)\frac{1}{|Ny|^\beta}\,dy|^2\right)^\frac{1}{2}\\
&\lesssim&\left[\sum_NN^{\frac{3n+\gamma}{2}}N^{-n}\int_{|y|\leq
N^{-1}}|f(x-y)|^2\,dy\right.\\
& & \left.+\sum_NN^{\frac{3n+\gamma}{2}}\left(\int_{|y|>
N^{-1}}\frac{|f(x-y)|^2}{|y|^\alpha}\,dy\right)\left(\int_{|y|>
N^{-1}}\frac{|y|^\alpha}{|Ny|^{2\beta}}\,dy\right)\right]^\frac{1}{2},
\end{eqnarray*}
where $\alpha$ is decided later.\\
Note that
\begin{eqnarray*}
\sum_NN^{\frac{n+\gamma}{2}}\chi_{\{|y|\leq
N^{-1}\}}(y)&\lesssim&\sum\limits_{|y|\leq
N^{-1}}N^{\frac{n+\gamma}{2}}\\
&\lesssim&|y|^{-\frac{n+\gamma}{2}}\\
\sum_NN^{\frac{3n+\gamma}{2}}\left(\int_{|y|>
N^{-1}}\frac{|y|^\alpha}{|Ny|^{2\beta}}\,dy\right)\chi_{\{|y|>
N^{-1}\}}(y)&\lesssim&\sum\limits_{|y|>
N^{-1}}N^{\frac{3n+\gamma}{2}}N^{-2\beta}N^{-(n+\alpha-2\beta)}\\&\lesssim&|y|^{\alpha-\frac{n+\gamma}{2}},
\end{eqnarray*}
where choosing $\alpha$ and $\beta$ to satisfy that
$n+\alpha-2\beta<0,\ \frac{\gamma+n}{2}-\alpha<0$,\\
 we obtain
\begin{eqnarray*}
S(|\nabla|^{-\frac{n-\gamma}{4}} f)(x)&\lesssim&\left(\int_{|y|\leq
N^{-1}}\frac{|f(x-y)|^2}{|y|^\frac{n+\gamma}{2}}\,dy+\int_{|y|>
N^{-1}}\frac{|f(x-y)|^2}{|y|^\frac{n+\gamma}{2}}\,dy\right)^\frac{1}{2}\\
&\thicksim&\left(\int\frac{|f(x-y)|^2}{|y|^\frac{n+\gamma}{2}}\,dy\right)^{\frac{1}{2}}\\
&\thicksim&
\left[(|\nabla|^{-\frac{n-\gamma}{2}}|f|^2)(x)\right]^{\frac{1}{2}},
\end{eqnarray*}
and we complete the proof.
\end{proof}\par
Using interpolation and Young's inequality, we get
\[
\parallel u\parallel_{L^q}^q\lesssim\parallel \nabla
u\parallel_{L^2}^\frac{2(n-\gamma)}{2+n-\gamma}\parallel
|\nabla|^{-\frac{n-\gamma}{4}} u\parallel
_{L^4}^\frac{8}{2+n-\gamma}\lesssim\varepsilon\parallel \nabla
u\parallel_{L^2}^2+C(\varepsilon)\parallel
|\nabla|^{-\frac{n-\gamma}{4}} u\parallel _{L^4}^4,
\]
where $q=\frac{2(4+n-\gamma)}{2+n-\gamma}$. Then,
\[
\parallel
|\nabla|^{-\frac{n-\gamma}{4}} u\parallel _{L^4}^4\geq
c(\varepsilon)\parallel u\parallel_{L^q}^q-c(\varepsilon)\parallel
\nabla u\parallel_{L^2}^2.
\]
On the other hand, In view of
\[
\frac{|\lambda_2|}{4}\parallel
|\nabla|^{-\frac{n-\gamma}{2}}|u|^2\parallel_{L^2}^2\gtrsim\parallel
|\nabla|^{-\frac{n-\gamma}{4}} u\parallel _{L^4}^4\geq
c(\varepsilon)\parallel u\parallel_{L^q}^q-c(\varepsilon)\parallel
\nabla u\parallel_{L^2}^2
\] from \eqref{4.6},
and
\[
\parallel u\parallel_{L^{p+2}}^{p+2}\leq C(\delta)\parallel
u\parallel_{L^2}^2+\delta\parallel u\parallel_{L^q}^q
\] from \eqref{4.5},
We have
\[
E\geq\frac{1}{2}\int|\nabla u|^2\,dx+c(\varepsilon)\parallel
u\parallel_{L^q}^q-c(\varepsilon)\parallel \nabla
u\parallel_{L^2}^2-\frac{|\lambda_1|}{p+2}\delta\parallel
u\parallel_{L^q}^q-\frac{|\lambda_1|}{p+2}C(\delta)\parallel
u\parallel_{L^2}^2.
\]
Choosing $\varepsilon$ and $\delta=\delta(\varepsilon)$ be small
enough, we obtain $$\parallel u(t)\parallel_{\dot{H}_x^1}\leq
C(E,M).$$ If $p\geq\frac{4}{2+n-\gamma}$, notice $\lambda_2>0$ and
\eqref{4.1}, using the identical method which is applied for case
(2), under the conditions of case \eqref{case 3} we have
$$\parallel u(t)\parallel_{\dot{H}_x^1}\leq C(E,M).$$

For the case \eqref{case 4}, by using \eqref{4.1},\ \eqref{4.2} and
Young's inequality, we have
\begin{eqnarray*}
E&\geq&\frac{1}{2}\int|\nabla
u|^2\,dx-\frac{|\lambda_1|}{p+2}C_R\parallel
u\parallel_{L^2}^{\frac{4-(n-2)p}{2}}\parallel\nabla
u\parallel_{L^2}^{\frac{np}{2}}-\frac{|\lambda_2|}{4}C_W\parallel
u\parallel_{L^2}^{4-\gamma}\parallel\nabla
u\parallel_{L^2}^\gamma \\
&\geq&\frac{1}{2}\int|\nabla
u|^2\,dx-\frac{|\lambda_1|}{p+2}C_R\delta\parallel\nabla
u\parallel_{L^2}^2-\frac{|\lambda_2|}{4}C_W\delta\parallel\nabla
u\parallel_{L^2}^2-C(M).
\end{eqnarray*}
Chosen $\delta$ to be sufficiently small, we obtain $$\parallel
u(t)\parallel_{\dot{H}_x^1}\leq C(E,M).$$

\subsection{Global well-posedness}
\ \quad In this subsection, we'll complete the proof of Theorem
\ref{1.1}. As mentioned above, when both nonlinearities are
$\dot{H}_x^1$-subcritical, according to the Proposition
$\ref{pro3.1}$, the prior control on the kinetic and the
conservation of mass, we can conclude the unique strong solution $u$
to \eqref{1} is a global solution. More precisely, in this
situation, we can find $T=T(\parallel u_0\parallel_{H^1_x})$ such
that \eqref{1} admits a unique strong solution $u\in
S^1([-T,T]\times\mathbb{R}^n)$ and
\[
\parallel u\parallel_{S^1([-T,T]\times\mathbb{R}^n)}\leq C(E,M).
\]
If we subdivide the interval $I$ into subintervals of length $T$,
deriving the corresponding $S^1-$bounds on each of these
subintervals, and at last summing these dominate together, then we
can get the bound \eqref{1.6}.\par
 When one of the nonlinearities is $\dot{H}_x^1$-critical, we view
 the other nonlinearity as a perturbation to the energy-critical
 $\textsl{NLS}$, which is globally wellposed, \cite{2,3,7,11,15,17}. Here we only discuss the case: $p=\frac{4}{n-2},\
 0<\gamma<\min{\{n,4\}}$, the same method can be used for the other
 case: $0<p<\frac{4}{n-2},\
 \gamma=4\ \mbox{with }n\geq 5$. Through the proof, we can find by
 Strichartz estimates and H\"older inequality, we need the
 coefficient of subcritical nonlinearity including $T=T(E,M)$, which
 will be required small in order to apply the standard continuity
 argument. So our approach don't fit the case that both
 nonlinearities are $\dot{H}_x^1$-critical.\par
   Let  $v$ be the unique strong global solution to the
   energy-critical equation \eqref{4} with initial datum $v_0=u_0$ at time
   $t=0$. By the main result in \cite{2,7,11,15,17}, we know that such a $v$ exists and

   \begin{equation} \label{4.8}
   \parallel
   v\parallel_{\dot{S}^1(\mathbb{R}\times\mathbb{R}^n)}\leq C(\parallel
   u_0\parallel_{\dot{H}_x^1}).
   \end{equation}
 Furthermore, by Lemma $\ref{lemma 3.6}$, we also have
 \[
   \parallel
 v\parallel_{\dot{S}^0(\mathbb{R}\times\mathbb{R}^n)}\leq C(\parallel u_0\parallel_{\dot{H}_x^1})\parallel
 u_0\parallel_{L^2}\leq C(E,M).
 \]
 By time reversal symmetry, it suffices to solves the problem
 forward in time. By \eqref{4.8}, split $\mathbb{R}^+$ into
 $J=J(E,\eta)$ subintervals $I_j=[t_j,t_{j+1}]$ such that
 \begin{equation}\label{4.9}
  \parallel v\parallel_{\dot{B}^1(I_j)}\thicksim\eta
 \end{equation}
 for some small $\eta$ to be chosen later.\par
 We may assume that there exits $J\,'<J$ such that for any $0\leq j\leq J\,'-1,\ [0,T]\cap
 I_j\neq\emptyset$.\\
 Thus, we can write$\hspace{2cm}[0,T]=\bigcup\limits_{j=0}^{J'-1}\left([0,T]\cap
 I_j\right)$.\\
 According to the Strichartz estimate, Sobolev embedding and \eqref{4.9}, we have
 the free evolution $e^{i(t-t_j)\Delta}v(t_j)$ is small on $I_j$
 \begin{eqnarray*}
   \parallel e^{i(t-t_j)\Delta}v(t_j)\parallel_{\dot{B}^1(I_j)}&\leq&\parallel
   v\parallel_{\dot{B}^1(I_j)}+\parallel\nabla\left(|v|^{\frac{4}{n-2}}v\right)\parallel_{L_{t,x}^\frac{2(n+2)}{n+4}(I_j\times\mathbb{R}^n)}\\
   &\leq& \parallel
   v\parallel_{\dot{B}^1(I_j)}+C\parallel
   v\parallel_{\dot{X}^1(I_j)}^{\frac{n+2}{n-2}}\\
   &\leq&\parallel
   v\parallel_{\dot{B}^1(I_j)}+C\parallel
   v\parallel_{\dot{B}^1(I_j)}^{\frac{n+2}{n-2}}\\
   &\leq&\eta+C\eta^{\frac{n+2}{n-2}}.
 \end{eqnarray*}
Thus, taking $\eta$ sufficiently small, for any $0\leq j\leq
J\,'-1$, we obtain
\[
\parallel
e^{i(t-t_j)\Delta}v(t_j)\parallel_{\dot{B}^1(I_j)}\leq 2\eta
\]
On the interval $I_0$, recalling that $u(0)=v(0)=u_0$, we estimate
\begin{eqnarray*}
\parallel u\parallel_{\dot{X}^1(I_0)}&\leq&\parallel
e^{it\Delta}u_0\parallel_{\dot{X}^1(I_0)}+C|I_0|^\alpha\parallel
u\parallel_{\dot{Y}^1(I_0)}^3+C\parallel
   u\parallel_{\dot{X}^1(I_0)}^{\frac{n+2}{n-2}},\\
\parallel u\parallel_{\dot{Y}^1(I_0)}&\leq&\parallel
e^{it\Delta}u_0\parallel_{\dot{Y}^1(I_0)}+C|I_0|^\alpha\parallel
u\parallel_{\dot{Y}^1(I_0)}^3+C\parallel
   u\parallel_{\dot{X}^1(I_0)}^{\frac{n+2}{n-2}},
\end{eqnarray*}
then \begin{eqnarray*}\parallel
u\parallel_{\dot{B}^1(I_0)}&\leq&\parallel
e^{it\Delta}u_0\parallel_{\dot{B}^1(I_0)}+CT^\alpha\parallel
u\parallel_{\dot{B}^1(I_0)}^3+C\parallel
   u\parallel_{\dot{B}^1(I_0)}^{\frac{n+2}{n-2}}\\
   &\leq&2\eta+CT^\alpha\parallel
u\parallel_{\dot{B}^1(I_0)}^3+C\parallel
   u\parallel_{\dot{B}^1(I_0)}^{\frac{n+2}{n-2}},
   \end{eqnarray*}
where $\alpha=\min{\{1,\ 2-\frac{\gamma}{2}\}}$.
\par
 Assuming $\eta$ and T are sufficiently small, a standard continuity
 argument then yields
 \[
 \parallel
u\parallel_{\dot{B}^1(I_0)}\leq4\eta.
 \]
In order to use Lemma $\ref{lemma3.4}$, we notice that \eqref{3.31}
holds on $I:=I_0$ for $M_0:=4C\eta$, \eqref{3.29} holds for
$E_0:=C(E,M)$. Also, \eqref{3.30} holds with $E'_0=0$. We only prove
that the error, which in this case is the second nonlinearity, is
sufficiently small.\\
In fact
\[
\parallel\nabla
e\parallel_{\dot{N}^0(I_0\times\mathbb{R}^n)}\lesssim
T^\alpha\parallel u\parallel_{\dot{Y}^1(I_0)}^3\lesssim
T^\alpha\parallel u\parallel_{\dot{B}^1(I_0)}^3\lesssim
T^\alpha\eta^3.
\]
We see that by choosing T sufficiently small, we get
\[
\parallel\nabla
e\parallel_{\dot{N}^0(I_0\times\mathbb{R}^n)}<\epsilon,
\]
where $\epsilon=\epsilon(E,M)$ is a small constant to be chosen
later. Thus, taking $\epsilon$ sufficiently small, the hypothesis of
Lemma $\ref{lemma3.4}$ are satisfied, which implies that
\begin{equation}\label{4.10}
\parallel u-v\parallel_{\dot{S}^1(I_0\times\mathbb{R}^n)}\leq
C(E,M)\epsilon^c
\end{equation}
for a small positive constant c which depends only on the dimension
$n$.\\
 Strichartz estimates and \eqref{4.10} imply
\begin{eqnarray}
\parallel u(t_1)-v(t_1)\parallel_{\dot{H}_x^1}&\leq&
C(E,M)\epsilon^c,\label{4.11}\\
\parallel e^{i(t-t_1)\Delta}\left(u(t_1)-v(t_1)\right)\parallel_{\dot{B}^1(I_1)}&\leq&
C(E,M)\epsilon^c.\label{4.12}
\end{eqnarray}
By using \eqref{4.11},\ \eqref{4.12} and Strichartz estimates, we
can get
\begin{eqnarray*}
\parallel u\parallel_{\dot{B}^1(I_1)}&\leq&\parallel
e^{i(t-t_1)\Delta}u(t_1)\parallel_{\dot{B}^1(I_1)}+CT^\alpha\parallel
u\parallel_{\dot{B}^1(I_1)}^3+C\parallel
u\parallel_{\dot{B}^1(I_1)}^\frac{n+2}{n-2}\\
&\leq&\parallel
e^{i(t-t_1)\Delta}v(t_1)\parallel_{\dot{B}^1(I_1)}+\parallel
e^{i(t-t_1)\Delta}\left(u(t_1)-v(t_1)\right)\parallel_{\dot{B}^1(I_1)}\\
&&+CT^\alpha\parallel u\parallel_{\dot{B}^1(I_1)}^3+C\parallel
u\parallel_{\dot{B}^1(I_1)}^\frac{n+2}{n-2}\\
&\leq&2\eta+C(E,M)\epsilon^c+CT^\alpha\parallel
u\parallel_{\dot{B}^1(I_1)}^3+C\parallel
u\parallel_{\dot{B}^1(I_1)}^\frac{n+2}{n-2}.
\end{eqnarray*}
A standard continuity method then yields
 \[
 \parallel
u\parallel_{\dot{B}^1(I_0)}\leq4\eta
 \]
provided $\epsilon$ is chosen sufficiently small depending on E and
M, which amounts to taking T sufficiently small depending on E and
M. We apply Lemma $\ref{lemma3.4}$ again on $I:=I_1$ to obtain
\[
\parallel u-v\parallel_{\dot{S}^1(I_1\times\mathbb{R}^n)}\leq
C(E,M)\epsilon^c.
\]
By induction argument, for every $0\leq j\leq J\,'-1$, we obtain
\begin{equation}\label{4.13}
\parallel
u\parallel_{\dot{B}^1(I_j)}\leq4\eta
\end{equation}
provided $\epsilon$ (and hence $T$) is sufficiently small depend on
$E$ and $M$. Adding \eqref{4.13} over all $0\leq j\leq J\,'-1$ and
recalling that $J\,'<J=J(E,M)$, we obtain
\begin{equation}
\parallel\label{4.14}
u\parallel_{\dot{B}^1([0,T])}\leq4J'\eta\leq C(E,M).
\end{equation}
Using Strichartz estimates, \eqref{2.4},\eqref{4.14} and $T=T(E,M)$,
we get
\begin{eqnarray}\label{4.15}
\parallel u\parallel_{\dot{S}^1([0,T]\times\mathbb{R}^n)}\lesssim
\parallel
u_0\parallel_{\dot{H}^1_x}+T^\alpha\parallel
u\parallel_{\dot{B}^1([0,T])}^3+\parallel
   u\parallel_{\dot{B}^1([0,T])}^{\frac{n+2}{n-2}}\leq C(E,M).
\end{eqnarray}
Similarly,
\begin{eqnarray*}
\parallel u\parallel_{\dot{S}^0([0,T]\times\mathbb{R}^n)}&\lesssim&
\parallel
u_0\parallel_{L^2_x}+T^\alpha\parallel
u\parallel_{\dot{B}^1([0,T])}^2\parallel
u\parallel_{\dot{Y}^0([0,T])}+\parallel
   u\parallel_{\dot{B}^1([0,T])}^{\frac{4}{n-2}}\parallel
u\parallel_{\dot{X}^0([0,T])}\\
&\lesssim&M^\frac{1}{2}+C(E,M)\parallel
u\parallel_{\dot{B}^1([0,T])}^2\parallel
u\parallel_{\dot{S}^0([0,T])}+\parallel
   u\parallel_{\dot{B}^1([0,T])}^{\frac{4}{n-2}}\parallel
u\parallel_{\dot{S}^0([0,T])}.
\end{eqnarray*}
Split $[0,T]$ into $N=N(E,M,\delta)$ subintervals $J_k$ such that
\begin{equation*}
  \parallel u\parallel_{\dot{B}^1(J_k)}\thicksim\delta
 \end{equation*}
for some small constant $\delta>0$ to be chosen later. Thus we get
\[
\parallel u\parallel_{\dot{S}^0(J_k\times\mathbb{R}^n)}\lesssim
M^\frac{1}{2}+C(E,M)\delta^2\parallel
u\parallel_{\dot{S}^0(J_k\times\mathbb{R}^n)}+\delta^{\frac{4}{n-2}}\parallel
u\parallel_{\dot{S}^0(J_k\times\mathbb{R}^n)}.
\]
Choosing $\delta$ sufficiently small, a standard continuity method
then implies
\[
\parallel
u\parallel_{\dot{S}^0(J_k\times\mathbb{R}^n)}\leq C(E,M).
\]
Adding these bounds over all subintervals $J_k$, we get
\begin{equation}\label{4.16}
\parallel u\parallel_{\dot{S}^0([0,T]\times\mathbb{R}^n)}\leq
C(E,M).
\end{equation}
Combine \eqref{4.15} and \eqref{4.16}, we get
\[
\parallel u\parallel_{S^1([0,T]\times\mathbb{R}^n)}\leq C(E,M).
\]
This completes the proof of Theorem \ref{1.1}.
 \section{Scattering results}
 \subsection{The interaction Morawetz inequality}
 \begin{prop}\label{5.1}(Morawetz control)
 Let $I$ be a compact interval, $\lambda_1$ and $\lambda_2$ are
 positive real numbers, and $u$ a solution to \eqref{1.1} on the
 slab $I\times\mathbb{R}^n$. Then
 \begin{equation}
\parallel u\parallel_{Z(I)}\lesssim\parallel
u\parallel_{L^\infty_tH_x^1(I\times\mathbb{R}^n)}.
 \end{equation}
 \end{prop}
 We will derive Proposition $5.1$ from the following:
 \begin{prop}(General interaction Morawetz inequality)
 \begin{eqnarray}
 &&-(n-1)\int_I\int_{\mathbb{R}^n}\int_{\mathbb{R}^n}\Delta(\frac{1}{|x-y|})|u(y)|^2|u(x)|^2\,dxdydt\nonumber\\
 &&
 +2\int_I\int_{\mathbb{R}^n}\int_{\mathbb{R}^n}|u(t,y)|^2\frac{x-y}{|x-y|}\{\textit{N},u\}_p(t,x)\,dxdydt\label{5.2}\\
 &\leq&4\parallel u\parallel_{L^\infty_tL_x^2(I\times\mathbb{R}^n)}^3\parallel\nabla u\parallel_{L^\infty_tL_x^2(I\times\mathbb{R}^n)}
 \nonumber\\
 &&+4\int_I\int_{\mathbb{R}^n}\int_{\mathbb{R}^n}|\{\textit{N},u\}_m(t,y)u(t,x)\nabla
 u(t,x)|\,dxdydt,\nonumber
 \end{eqnarray}
 where $\textit{N}:=\lambda_1|u|^pu+\lambda_2(|x|^{-\gamma}\ast
|u|^2)u,\ \{f,g\}_p:=Re(f\nabla\bar{g}-g\nabla\bar{f}),\
\{f,g\}_m=Im\{f\bar{g}\}$.
 \end{prop}
The proof can be found in \cite{24}.\par
 Note that, in  particular $\textit{N}:=\lambda_1|u|^pu+\lambda_2(|x|^{-\gamma}\ast
|u|^2)u$, we have \[\{\textit{N},u\}_m=0,\quad
\{\textit{N},u\}_p=-\frac{\lambda_1p}{p+2}\nabla(|u|^{p+2})-\lambda_2Re\left\{\nabla(|x|^{-\gamma}\ast|u|^2)|u|^2\right\}.\]
Next we'll show \eqref{5.2} is positive, then we obtain
\begin{equation}\label{5.3}
-\int_I\int_{\mathbb{R}^n}\int_{\mathbb{R}^n}\Delta(\frac{1}{|x-y|})|u(y)|^2|u(x)|^2\,dxdydt\leq\parallel
u\parallel_{L^\infty_tH_x^1(I\times\mathbb{R}^n)}^4.
\end{equation}
In dimension $n=3$, we have $-\Delta(\frac{1}{|x|})=4\pi\delta$, so
\eqref{5.3} yields
\[
\parallel u\parallel_{L_{t,x}^4(I\times\mathbb{R}^3)}^4\lesssim\parallel
u\parallel_{L^\infty_tH_x^1(I\times\mathbb{R}^3)}^4,
\]
which proves the Proposition $\ref{5.1}$.\par
 In dimension $n\geq4$, we have
 $-\Delta(\frac{1}{|x|})=\frac{n-3}{|x|^3}$, so \eqref{5.3}
 yields
 \begin{equation}\label{5.4}
 \parallel
 |\nabla|^{-\frac{n-3}{2}}|u|^2\parallel_{L_{t,x}^2(I\times\mathbb{R}^n)}\lesssim\parallel
u\parallel_{L^\infty_tH_x^1(I\times\mathbb{R}^n)}^2.
 \end{equation}
From Lemma $\ref{lemma4.3}$ and the above inequality, we have
\begin{equation}\label{5.5}
 \parallel
 |\nabla|^{-\frac{n-3}{4}}u\parallel_{L_{t,x}^4(I\times\mathbb{R}^n)}\lesssim\parallel
u\parallel_{L^\infty_tH_x^1(I\times\mathbb{R}^n)}.
\end{equation}
Proposition \eqref{5.1} follows by interpolation between \eqref{5.5}
and the bound on the kinetic energy
\[
\parallel\nabla u\parallel_{L^\infty_tL_x^2}\lesssim
E^\frac{1}{2},
\]
which is an immediate consequence of the conservation of energy when
both nonlinearities are defocusing. Note that
\begin{eqnarray*}
&&\int_I\int_{\mathbb{R}^n}\int_{\mathbb{R}^n}|u(t,y)|^2\frac{x-y}{|x-y|}\{\textit{N},u\}_p(t,x)\,dxdydt\\
&=&-\int_I\int_{\mathbb{R}^n}\int_{\mathbb{R}^n}|u(t,y)
|^2\frac{x-y}{|x-y|}\frac{\lambda_1p}{p+2}\nabla(|u|^{p+2})\,dxdydt\\
&&-\lambda_2Re\int_I\int_{\mathbb{R}^n}\int_{\mathbb{R}^n}|u(t,y)|^2\frac{x-y}{|x-y|}\left\{\nabla(|x|^{-\gamma}\ast|u|^2)|u|^2\right\}\,dxdydt\\
&=&(I)+(\amalg).
\end{eqnarray*}
For $(I)$, we have
\[
-\int_I\int_{\mathbb{R}^n}\int_{\mathbb{R}^n}|u(t,y)
|^2\frac{x-y}{|x-y|}\frac{\lambda_1p}{p+2}\nabla(|u|^{p+2})\,dxdydt=(n-1)\frac{\lambda_1p}{p+2}\int_I\int_{\mathbb{R}^n}\frac{|u(t,y)
|^2|u(t,x) |^{p+2}}{|x-y|}\,dxdydt.
\]
Note that $\lambda_1>0$, we get $(I)$ is positive.\\
For $(\amalg)$, we define $h(x)=\int_{\mathbb{R}^n}|u(t,y)
|^2\frac{x-y}{|x-y|}\,dy$, then we have
\begin{eqnarray*}
(\amalg)&=&-\lambda_2Re\int_I\int_{\mathbb{R}^n}h(x)\left\{\nabla(|x|^{-\gamma}\ast|u|^2)|u|^2\right\}\,dxdt\nonumber\\
&=&\lambda_2\gamma
Re\int_I\int_{\mathbb{R}^n}\int_{\mathbb{R}^n}\frac{1}{|x-z|^{\gamma+1}}\frac{x-z}{|x-z|}|u(t,z)|^2|u(t,x)|^2h(x)\,dxdzdt\\
&=&\frac{1}{2}\lambda_2\gamma
Re\int_I\int_{\mathbb{R}^n}\int_{\mathbb{R}^n}\frac{1}{|x-z|^{\gamma+2}}|u(t,z)|^2|u(t,x)|^2[(x-z)(h(x)-h(z))]\,dxdzdt.
\end{eqnarray*}
Notice that
\begin{equation}
(x-z)(h(x)-h(z))=(x-z)\int_{\mathbb{R}^n}|u(t,y)
|^2\left(\frac{x-y}{|x-y|}-\frac{z-y}{|z-y|}\right)\,dy\label{5.6}\end{equation}
and denote $a:=x-y,\quad b:=z-y$, then, we have
$(5.6)=\int_{\mathbb{R}^n}|u(t,y)
|^2(a-b)(\frac{a}{|a|}-\frac{b}{|b|})\,dy$.\\
Since
$(a-b)(\frac{a}{|a|}-\frac{b}{|b|})=(|a||b|-ab)(\frac{1}{|a|}+\frac{1}{|b|})\geq0$
and $\lambda_2>0$, thus $(\amalg)$ is positive, so we show
\eqref{5.2} is positive.
\begin{remark}
When $n=2$, we don't know whether $-\Delta(\frac{1}{|x|})$ is
positive or not. However, J. Colliander, M. Grillakis and N.
Tzirakis use a refined  tensor product approach to prove that
\eqref{5.4} also holds when $n=2$(\cite{13,14}). Then the
corresponding \eqref{5.1} and \eqref{2.6} also exist, we can use the
same approach which used in Section 5.3 to show the scattering of
the power type. However, the corresponding \eqref{2.7} don't hold.
Since we need $\gamma>2$, but in this case $\gamma<n=2$. So the
scattering of the Hartree type can't be gotten.
\end{remark}
\subsection{Global bounds in the case: $p=\frac{4}{n},\ 2<\gamma<\min{\{n,4\}}$ and $\lambda_1,\lambda_2>0$
or $\frac{4}{n}<p<\frac{4}{n-2},\ \gamma=2$ and
$\lambda_1,\lambda_2>0$} \ \quad The approaches for both cases are
the same, we settle the first case and the same method can be used
for the other one. Without loss of generality, let
$\lambda_1=\lambda_2=1$.\par We view the second nonlinearity as a
perturbation to \eqref{5}. By using Proposition \eqref{5.1}, and the
conservation of energy and mass, we get
\[
\parallel u\parallel_{Z(\mathbb{R})}\lesssim\parallel
u\parallel_{L^\infty_tH_x^1(\mathbb{R}\times\mathbb{R}^n)}\leq
C(E,M).
\]
Split $\mathbb{R}$ into $J=J(E,M,\varepsilon)$ subintervals $I_j,\
0\leq j\leq J-1$, such that
\[
\parallel u\parallel_{Z(I_j)}\sim\varepsilon,
\]
where $\varepsilon$ is a small positive constant to be chosen
later.\\
On the slab $I\times\mathbb{R}^n$, we define:
\[
\dot{\widetilde{X}^0}(I):=L_t^{2+\frac{1}{\theta}}L_x^{\frac{2n(2\theta+1)}{n(2\theta+1)-4\theta}}(I\times\mathbb{R}^n)\cap
V(I),
\]
where $\theta$ is introduced in Lemma $\ref{lemma2.6}$. Then on each
$I_j\,(0\leq j\leq J-1)$, by \eqref{2.8} we have
\begin{eqnarray}
\parallel(|x|^{-\gamma}\ast|u|^2)u\parallel_{\dot{N}^0(I_j\times\mathbb{R}^n)}&\lesssim&\parallel
u\parallel_{L_t^{2+\frac{1}{\theta}}L_x^{\frac{2n(2\theta+1)}{n(2\theta+1)-4\theta}}(I_j\times\mathbb{R}^n)}\parallel
u\parallel_{Z(I_j)}^{\frac{n+1}{2(2\theta+1)}}\parallel
 u\parallel_{L_t^\infty
 H_x^1(I_j\times\mathbb{R}^n)}^{\beta_1(\theta)+\beta_2(\theta)}\nonumber\\
 &\leq&C(E,M)\varepsilon^c\parallel
 u\parallel_{\dot{\widetilde{X}^0}(I_j)},\label{5.7}
\end{eqnarray}
where $c=\frac{n+1}{2(2\theta+1)}$.\par
 In what follow, we fix an interval $I_{j_0}=[a,b]$ and prove that $u$
 obeys good Strichartz estimates on the slab
 $I_{j_0}\times\mathbb{R}^n$. Let $v$ be a solution to
\begin{eqnarray*}
 \left\{
\begin{array}{ll}
iv_t+\Delta v=|v|^{\frac{4}{n}}v,\\
v(a)=u(a).
\end{array}
\right.
\end{eqnarray*}
As this initial value problem is globally well-posedness in $H_x^1$,
and by Assumption 1.1 and Lemma \ref{lemma 3.6}, the unique solution
$v$ satisfies
\[
\parallel v\parallel_{\dot{S}^0(\mathbb{R}\times\mathbb{R}^n)}\leq
C(M).
\]
Subdivide $\mathbb{R}$ into $K=K(M,\eta)$ subinterval $J_k$ such
that on each $J_k$
\begin{equation}
\parallel v\parallel_{\dot{\widetilde{X}^0}(J_k)}\sim\eta
\end{equation}
for a small constant $\eta>0$ to be chosen later.\par
 We are only
interested in the subintervals $J_k=[t_k,t_{k+1}]$ which have a
nonempty intersection with $I_{j_0}$. Without loss of generality,
assume that $[a,b]=\cup_{k=0}^{k'-1}J_k,
t_0=a,t_{k'}=b$.\\
On each $J_k$, by Strichartz estimates and \eqref{5.4}, we get
\[
\parallel
e^{i(t-t_k)\Delta}v(t_k)\parallel_{\dot{\widetilde{X}^0}(J_k)}\leq\parallel
v\parallel_{\dot{\widetilde{X}^0}(J_k)}+C\parallel
|v|^{\frac{4}{n}}v\parallel_{\dot{N}^0(J_k\times\mathbb{R}^n)}\leq\eta+C\parallel
v\parallel_{V(J_k)}^{1+\frac{4}{n}}\leq\eta+C\eta^{1+\frac{4}{n}}.
\]
Choosing $\eta$ sufficiently small, we get \begin{equation}
\parallel
e^{i(t-t_k)\Delta}v(t_k)\parallel_{\dot{\widetilde{X}^0}(J_k)}\leq
2\eta.\label{5.9}
\end{equation}
Next, we will use Lemma $\ref{lemma3.5}$ to obtain an estimate on
the $S^1-$norm of $u$ on $I_{j_0}\times\mathbb{R}^n$. On the
interval $J_0$, recalling that $u(t_0)=v(t_0)$, by Strichartz
estimates, \eqref{5.7} and \eqref{5.9},
\begin{eqnarray*}
\parallel u\parallel_{\dot{\widetilde{X}^0}(J_0)}&\leq&\parallel
e^{i(t-t_0)\Delta}u(t_0)\parallel_{\dot{\widetilde{X}^0}(J_0)}+C\parallel
u\parallel_{\dot{\widetilde{X}^0}(J_0)}^{1+\frac{4}{n}}+C(E,M)\varepsilon^c\parallel
u\parallel_{\dot{\widetilde{X}^0}(J_0)}\\
&\leq&2\eta+C\parallel
u\parallel_{\dot{\widetilde{X}^0}(J_0)}^{1+\frac{4}{n}}+C(E,M)\varepsilon^c\parallel
u\parallel_{\dot{\widetilde{X}^0}(J_0)}.
\end{eqnarray*}
By a standard continuity argument yields
\[
\parallel u\parallel_{\dot{\widetilde{X}^0}(J_0)}\leq4\eta
\]
provided $\eta$ and $\varepsilon$ are chosen sufficiently small. In
order to use Lemma \ref{lemma3.5}, we notice that \eqref{3.39} holds
on $I:=J_0$ for $L_0:=4\eta,$\ \eqref{3.38} holds with $M'_0=0$. We
only show that the error is sufficiently small. In fact, from
\[
\parallel
e\parallel_{\dot{N}^0(J_0\times\mathbb{R}^n)}\leq
C(E,M)\varepsilon^c\parallel
u\parallel_{\dot{\widetilde{X}^0}(J_0)}\leq C(E,M)\eta\varepsilon^c,
\] and
choosing $\varepsilon$ to be sufficiently small, we obtain
\[
\parallel
u-v\parallel_{\dot{S}^0(J_0\times\mathbb{R}^n)}\leq\varepsilon^\frac{c}{2}.
\]
From Strichartz estimates, we have
\begin{eqnarray}
\parallel
u(t_1)-v(t_1)\parallel_{L_x^2}&\leq&\varepsilon^\frac{c}{2}\nonumber\\
\parallel
e^{i(t-t_1)\Delta}(u(t_1)-v(t_1))\parallel_{\dot{\widetilde{X}^0}(J_1)}&\lesssim&\varepsilon^\frac{c}{2}.\label{5.10}
\end{eqnarray}
On the other hand,
\begin{eqnarray*}
\parallel
u\parallel_{\dot{S}^1(J_0\times\mathbb{R}^n)}&\lesssim&\parallel
u(a)\parallel_{\dot{H}_x^1}+\parallel
u\parallel_{V(J_0)}^\frac{4}{n}\parallel
u\parallel_{\dot{S}^1(J_0\times\mathbb{R}^n)}+\parallel(|x|^{-\gamma}\ast|u|^2)u\parallel_{\dot{N}^1(I\times\mathbb{R}^n)}\\
&\lesssim&C(E)+(4\eta)^\frac{4}{n}\parallel
u\parallel_{\dot{S}^1(J_0\times\mathbb{R}^n)}+C(E,M)\varepsilon^c\parallel
u\parallel_{\dot{S}^1(J_0\times\mathbb{R}^n)}.
\end{eqnarray*}
Choosing $\eta$ and $\varepsilon$ sufficiently small, we have
\[
\parallel
u\parallel_{\dot{S}^1(J_0\times\mathbb{R}^n)}\leq C(E).
\]
On the intervals $J_1$, by Strichartz estimates, \eqref{5.7},\
\eqref{5.10}, we get
\begin{eqnarray*}
\parallel u\parallel_{\dot{\widetilde{X}^0}(J_1)}&\leq&\parallel
e^{i(t-t_1)\Delta}v(t_1)\parallel_{\dot{\widetilde{X}^0}(J_1)}+\parallel
e^{i(t-t_1)\Delta}(u(t_1)-v(t_1))\parallel_{\dot{\widetilde{X}^0}(J_1)}\\
&&+C\parallel
u\parallel_{\dot{\widetilde{X}^0}(J_1)}^{1+\frac{4}{n}}+C(E,M)\varepsilon^c\parallel
u\parallel_{\dot{\widetilde{X}^0}(J_1)}\\
&\leq&2\eta+\varepsilon^\frac{c}{2}+C\parallel
u\parallel_{\dot{\widetilde{X}^0}(J_1)}^{1+\frac{4}{n}}+C(E,M)\varepsilon^c\parallel
u\parallel_{\dot{\widetilde{X}^0}(J_1)}.
\end{eqnarray*}
Choosing $\eta$ and $\varepsilon$ sufficiently small, we obtain
\[
\parallel u\parallel_{\dot{\widetilde{X}^0}(J_1)}\leq4\eta.
\]
This implies that the error satisfies the condition of Lemma
$\ref{lemma3.5}$ on $J_1$. Choosing $\varepsilon$ sufficiently
small, and applying Lemma $\ref{lemma3.5}$ to derive
\[
\parallel
u-v\parallel_{\dot{S}^0(J_1\times\mathbb{R}^n)}\leq\varepsilon^\frac{c}{4}.
\]
The same arguments as before also yields
\[
\parallel
u\parallel_{\dot{S}^1(J_1\times\mathbb{R}^n)}\leq C(E).
\]
By the induction argument, for each $0\leq k\leq k'-1$, we get
\begin{eqnarray*}
\parallel
u-v\parallel_{\dot{S}^0(J_k\times\mathbb{R}^n)}&\leq&\varepsilon^\frac{c}{2^{k+1}},\\
\parallel
u\parallel_{\dot{S}^1(J_k\times\mathbb{R}^n)}&\leq& C(E).
\end{eqnarray*}
Adding these estimates over all the intervals $J_k$ which have a
nonempty intersection with $I_{j_0}$, we obtain
\begin{eqnarray*}
\parallel
u\parallel_{\dot{S}^0(I_{j_0}\times\mathbb{R}^n)}&\leq&\parallel
v\parallel_{\dot{S}^0(I_{j_0}\times\mathbb{R}^n)}+\sum\limits_{k=0}^{k'-1}\parallel
u-v\parallel_{\dot{S}^0(J_k\times\mathbb{R}^n)}\leq C(E,M)\\
\parallel
u\parallel_{\dot{S}^1(I_{j_0}\times\mathbb{R}^n)}&\leq&\sum\limits_{k=0}^{k'-1}\parallel
u\parallel_{\dot{S}^1(J_k\times\mathbb{R}^n)}\leq C(E,M).
\end{eqnarray*}
As the intervals $I_{j_0}$ was arbitrarily chosen, we obtain
\begin{eqnarray*}
\parallel
u\parallel_{\dot{S}^0(\mathbb{R}\times\mathbb{R}^n)}&\leq&\sum\limits_{j=0}^{J-1}\parallel
u\parallel_{\dot{S}^0(I_j\times\mathbb{R}^n)}\leq C(E,M)\\
\parallel
u\parallel_{\dot{S}^1(\mathbb{R}\times\mathbb{R}^n)}&\leq&\sum\limits_{j=0}^{J-1}\parallel
u\parallel_{\dot{S}^1(I_j\times\mathbb{R}^n)}\leq C(E,M),\\
\end{eqnarray*}
and hence
\[
\parallel
u\parallel_{S^1(\mathbb{R}\times\mathbb{R}^n)}\leq C(E,M).
\]
\subsection{Global bounds in the case: $\frac{4}{n}<p<\frac{4}{n-2},\ 2<\gamma<\min{\{n,4\}}$ and $\lambda_1,\lambda_2>0$}
The results were shown in \cite{19} with a more complicated
argument. We use a simpler proof which is used in \cite{24} that
relies on the interaction Morawetz estimate.\par
 By Proposition $\ref{5.1}$, we have
 \[
 \parallel u\parallel_{Z(\mathbb{R})}\lesssim\parallel
 u\parallel_{L^\infty_tH_x^1(\mathbb{R}\times\mathbb{R}^n)}\leq
 C(E,M)
 \]
Devide $\mathbb{R}$ into $J=J(E,M,\eta)$ subintervals
$I_j=[t_j,t_{j+1}]$ such that
\[
\parallel u\parallel_{Z({I}_j)}\sim\eta
\]
where $\eta>0$ be a small constant to be chosen later.\par
 By Strichartz estimates and Lemma \ref{lemma2.6}, on each $I_j$,
 we have
 \begin{eqnarray*}
\parallel
u\parallel_{S^1(I_j\times\mathbb{R}^n)}&\lesssim&\parallel
u(t_j)\parallel_{H_x^1}+\eta^{\frac{n+1}{2(2\theta+1)}}\parallel
u\parallel_{L_t^\infty
 H_x^1(I_j\times\mathbb{R}^n)}^{{\alpha_1(\theta)}+{\alpha_2(\theta)}}\parallel
u\parallel_{S^1(I_j\times\mathbb{R}^n)}\\
&&+\eta^{\frac{n+1}{2(2\theta+1)}}\parallel
 u\parallel_{L_t^\infty
 H_x^1(I_j\times\mathbb{R}^n)}^{\beta_1(\theta)+\beta_2(\theta)}\parallel
u\parallel_{S^1(I_j\times\mathbb{R}^n)}\\
&\lesssim&C(E,M)+\eta^{\frac{n+1}{2(2\theta+1)}}C(E,M)\parallel
u\parallel_{S^1(I_j\times\mathbb{R}^n)}\\
&&+\eta^{\frac{n+1}{2(2\theta+1)}}C(E,M)\parallel
u\parallel_{S^1(I_j\times\mathbb{R}^n)}.
 \end{eqnarray*}
Choosing $\eta$ sufficiently small, we have
\[
\parallel
u\parallel_{S^1(I_j\times\mathbb{R}^n)}\leq C(E,M).
\]
Summing these bounds over all intervals $I_j$, we obtain
\[
\parallel
u\parallel_{S^1(\mathbb{R}\times\mathbb{R}^n)}\leq\sum\limits_{j=0}^{J-1}\parallel
u\parallel_{S^1(I_j\times\mathbb{R}^n)}\leq C(E,M).
\]
\subsection{Global bounds in the case: $\frac{4}{n}<p<\frac{4}{n-2},\ \gamma=4$ with $n\geq5$ and $\lambda_1,\lambda_2>0$ or
 $p=\frac{4}{n-2},\ 2<\gamma<\min{\{n,4\}}$ and $\lambda_1,\lambda_2>0$}
\ \quad The approaches for both cases are the same, we show the
first case and the same method can be used for the other. On the
slab $I\times\mathbb{R}^n$, we define:
\[
\dot{\widetilde{Y}^0}(I):=L_t^{2+\frac{1}{\theta}}L_x^{\frac{2n(2\theta+1)}{n(2\theta+1)-4\theta}}(I\times\mathbb{R}^n)\cap
L_t^6L_x^\frac{6n}{3n-2}(I\times\mathbb{R}^n),
\]
where $\theta$ is introduced in Lemma \ref{lemma2.6}. Just replace
 $\dot{\widetilde{X}^0}(I)$ by $\dot{\widetilde{Y}^0}(I)$ that
appears in Subsection 5.2, Lemma \ref{lemma3.2} replace Lemma
\ref{lemma3.5}, apply the same approach that used in Subsection 5.2,
one can get
\[
\parallel u\parallel_{S^1(\mathbb{R}\times\mathbb{R}^n)}\leq C(E,M).
\]

\subsection{Global bounds in the case: $p=\frac{4}{n-2},\ \gamma=2$ and
$\lambda_1,\lambda_2>0$ or $p=\frac{4}{n},\ \gamma=4$ with $n\geq5$
and $\lambda_1,\lambda_2>0$} \ \quad The approaches for both cases
are the same, we settle the first case and the same method can be
used for the other one. Without loss of generality, let
$\lambda_1=\lambda_2=1$. The main idea is that we divide $u$ into
$u_{lo}$ and $u_{hi}$ by frequency, and compare the low frequency
with the $L_x^2$-critical $\textsl{NLS}$, at one time, compare the
high frequency with the $H_x^1$-critical $\textsl{NLS}$. At last, we
get the finite global Strichartz bounds in this case.\par
 We will need a series of small
parameters. More precisely, we will define
\[
0<\eta_3\ll\eta_2\ll\eta_1\ll1,
\]
where any $\eta_j$ is allowed to depend on the energy and the mass
as well as on any of the larger $\eta'$s. By Proposition \ref{5.1}
and conservation of energy and mass, we have
\[
\parallel u\parallel_{Z(\mathbb{R})}\leq C(E,M).
\]
Split $\mathbb{R}$ into $K=K(E,M,\eta_3)$ subintervals $J_k$ such
that on each slab $J_k\times\mathbb{R}^n$ we have
\begin{equation}\label{5.17}
\parallel u\parallel_{Z(J_k)}\sim\eta_3.
\end{equation}
Fix $J_{k_0}=[a,b]$, for every $t\in J_{k_0}$. We split
$u(t)=u_{lo}(t)+u_{hi}(t)$ where $u_{lo}(t):=P_{<\eta_2^{-1}}u(t),\
u_{hi}(t):=P_{\geq\eta_2^{-1}}u(t)$.\par
 On the slab $J_{k_0}\times\mathbb{R}^n$, we compare $u_{lo}(t)$ to
 the following $L_x^2$-critical Hartree $\textsl{NLS}$
 \begin{equation*}
\left\{\begin{array}{ll}
                   (i\partial_t+\Delta)v=(|x|^{-2}\ast|v|^2)v\\
                   v(a)=u_{lo}(a),
                   \end{array}
                   \right.
 \end{equation*}
which is globally well-posedness in $H_x^1$. Moreover, by Assumption
1.2, one has
\[
\parallel v\parallel_{U(\mathbb{R})}\leq C(\parallel
u_{lo}(a))\parallel_{L_x^2}\leq C(M).
\]
By Lemma \ref{lemma 3.6}, we have
\begin{eqnarray}
\parallel v\parallel_{\dot{S}^0(\mathbb{R}\times\mathbb{R}^n)}&\leq&
C(M),\label{5.18}\\
\parallel v\parallel_{\dot{S}^1(\mathbb{R}\times\mathbb{R}^n)}&\leq&
C(E,M).\label{5.19}
\end{eqnarray}
Divide $J_{k_0}=[a,b]$ into $J=J(M,\eta_1)$ subintervals
$I_j=[t_{j-1},t_j]$ with $t_0=a,t_J=b$, such that
\begin{equation}\label{5.20}
\parallel v\parallel_{U(I_j)}\sim\eta_1.
\end{equation}
By induction, we will establish that for each $j=1,\cdots,J$, we
have
\begin{equation}\label{5.21}
P(j):\left\{\begin{array}{lll}
               \parallel
               u_{lo}-v\parallel_{\dot{S}^0([t_0,t_j])}\leq\eta_2^{1-2\delta},\\
               \parallel u_{hi}\parallel_{\dot{S}^1(I_l)}\leq
               L(E),\qquad\mbox{for every} 1\leq l\leq j,\\
               \parallel u\parallel_{S^1([t_0,t_j])}\leq
               C(\eta_1,\eta_2),
               \end{array}
               \right.
\end{equation}
where $\delta>0$ is a small constant to be chosen later, and $L(E)$
is a large quantity to be chosen later which depends only on $E$(
not on any $\eta_j$). As the method of checking that \eqref{5.21}
holds for $j=1$ is similar to that of the induction step, i.e.
showing that $P(j)$ implies $P(j+1)$, we will only prove the
latter.\par
 Assume that \eqref{5.21} is true for some $1\leq j<J$. Then, we
 will show
\begin{equation}\label{5.22}
\left\{\begin{array}{lll}
               \parallel
               u_{lo}-v\parallel_{\dot{S}^0([t_0,t_{j+1}])}\leq\eta_2^{1-2\delta},\\
               \parallel u_{hi}\parallel_{\dot{S}^1(I_l)}\leq
               L(E),\qquad\mbox{for every} 1\leq l\leq j+1,\\
               \parallel u\parallel_{S^1([t_0,t_{j+1}])}\leq C(\eta_1,\eta_2)
               \end{array}
               \right..
\end{equation}
Let $\Omega_1$ be the set of all times $T\in I_{j+1}$ such that
\begin{eqnarray}
\parallel u_{lo}-v\parallel_{\dot{S}^0([t_0,T])}&\leq&\eta_2^{1-2\delta}\label{5.23},\\
\parallel u_{hi}\parallel_{\dot{S}^1([t_j,T])}&\leq&
               L(E)\label{5.24},\\
\parallel u\parallel_{S^1([t_0,T])}&\leq& C(\eta_1,\eta_2).\label{5.25}
\end{eqnarray}
In order to prove $\Omega_1=I_{j+1}$, we notice that $\Omega_1$ is
nonempty (as $t_j\in\Omega_1$) and closed (by Fatou). Let $\Omega_2$
be the set of all times $T\in I_{j+1}$ such that
\begin{eqnarray}
\parallel u_{lo}-v\parallel_{\dot{S}^0([t_0,T])}&\leq&2\eta_2^{1-2\delta},\label{5.26}\\
\parallel u_{hi}\parallel_{\dot{S}^1([t_j,T])}&\leq&
               2L(E),\label{5.27}\\
\parallel u\parallel_{S^1([t_0,T])}&\leq& 2C(\eta_1,\eta_2).\label{5.28}
\end{eqnarray}
We will show $\Omega_2\subset\Omega_1$, which will conclude the
argument.
\begin{lemma}
Let $T\in\Omega_2$. Then, the following properties holds:
\begin{eqnarray}
\parallel u_{lo}\parallel_{U(I)}&\lesssim&\eta_1\label{5.29},\\
\parallel u_{lo}\parallel_{\dot{S}^0([t_0,T]\times\mathbb{R}^n)}&\leq&
               C(M),\label{5.30}\\
\parallel u_{lo}\parallel_{W([t_j,T])}&\lesssim&\eta_2,\label{5.31}\\
\parallel
u_{lo}\parallel_{\dot{S}^1(I\times\mathbb{R}^n)}&\lesssim&E,\label{5.32}\\
\parallel
u_{lo}\parallel_{\dot{S}^1([t_0,T]\times\mathbb{R}^n)}&\lesssim&C(\eta_1)E,\label{5.33}\\
\parallel
u_{hi}\parallel_{\dot{S}^0(I\times\mathbb{R}^n)}&\lesssim&\eta_2L(E),\label{5.34}\\
\parallel
u_{hi}\parallel_{\dot{S}^0([t_0,T]\times\mathbb{R}^n)}&\lesssim&\eta_2C(\eta_1)L(E),\label{5.35}\\
\parallel
u_{hi}\parallel_{\dot{S}^1([t_0,T]\times\mathbb{R}^n)}&\lesssim&C(\eta_1)L(E),\label{5.36}\\
&&\hspace{-9cm}\mbox{where }I\in\{I_l,1\leq l\leq
j\}\cup\{[t_j,T]\}.\nonumber
\end{eqnarray}
\end{lemma}
\begin{proof}
Using \eqref{5.18}, \eqref{5.20}, \eqref{5.26}, and Bernstein
inequality, we have
\begin{eqnarray*}
\parallel u_{lo}\parallel_{U(I)}&\leq&\parallel
u_{lo}-v\parallel_{U(I)}+\parallel
v\parallel_{U(I)}\lesssim\eta_2^{(1-2\delta)}+\eta_1\lesssim\eta_1,\\
\parallel u_{lo}\parallel_{\dot{S}^0([t_0,T]\times\mathbb{R}^n)}&\leq&\parallel
u_{lo}-v\parallel_{\dot{S}^0([t_0,T]\times\mathbb{R}^n)}+\parallel
v\parallel_{\dot{S}^0([t_0,T]\times\mathbb{R}^n)}\lesssim\eta_2^{1-2\delta}+C(M)\leq
C(M),\\
\parallel
u_{hi}\parallel_{\dot{S}^0(I\times\mathbb{R}^n)}&\lesssim&\eta_2\parallel
u_{hi}\parallel_{\dot{S}^1(I\times\mathbb{R}^n)}\lesssim\eta_2L(E).
\end{eqnarray*}
Therefore, \eqref{5.29},\ \eqref{5.30} and \eqref{5.34} hold. In
view of $J=O(\eta_1^{-C})$, we get
\begin{eqnarray*}
\parallel
u_{hi}\parallel_{\dot{S}^1([t_0,T]\times\mathbb{R}^n)}&\lesssim&\sum\limits_{l=1}^j\parallel
u_{hi}\parallel_{\dot{S}^1(I_l\times\mathbb{R}^n)}+\parallel
u_{hi}\parallel_{\dot{S}^1([t_j,T]\times\mathbb{R}^n)}\leq
C(\eta_1)L(E)+\eta_2L(E)\leq C(\eta_1)L(E),\\
\parallel
u_{hi}\parallel_{\dot{S}^0([t_0,T]\times\mathbb{R}^n)}&\lesssim&\eta_2\parallel
u_{hi}\parallel_{\dot{S}^1([t_0,T]\times\mathbb{R}^n)}\leq\eta_2C(\eta_1)L(E).
\end{eqnarray*}
Hence, \eqref{5.35} and \eqref{5.36} hold. On the slab
$I\times\mathbb{R}^n$, $u_{lo}$ satisfies the equation
\[
u_{lo}(t)=e^{i(t-t_l)\Delta}u_{lo}(t_l)-i\int_{t_l}^te^{i(t-s)\Delta}P_{lo}\left(|u|^\frac{4}{n-2}u+(|x|^{-2}\ast|u|^2)u\right)(s)\,ds,
\]
where $0\leq l\leq j$. Then by Strichartz estimate
\[
\parallel
u_{lo}\parallel_{\dot{S}^1(I\times\mathbb{R}^n)}\lesssim\parallel
u_{lo}(t_l)\parallel_{\dot{H}^1_x}+\parallel
P_{lo}(|u|^\frac{4}{n-2}u)\parallel_{\dot{N}^1(I\times\mathbb{R}^n)}+\parallel
P_{lo}((|x|^{-2}\ast|u|^2)u)\parallel_{\dot{N}^1(I\times\mathbb{R}^n)}.
\]
By using Bernstein inequality, Lemma \ref{lemma},\ \eqref{5.17} and
\eqref{5.28}, we have
\[
\parallel
P_{lo}(|u|^\frac{4}{n-2}u)\parallel_{\dot{N}^1(I\times\mathbb{R}^n)}\lesssim\eta_2^{-1}\parallel
|u|^\frac{4}{n-2}u\parallel_{\dot{N}^0(I\times\mathbb{R}^n)}\lesssim\eta_2^{-1}\parallel
u\parallel_{Z(I)}^\rho\parallel
 u\parallel_{S^1(I\times\ast\mathbb{R}^n)}^{\frac{n+2}{n-2}-\rho}\lesssim\eta_2^{-1}\eta_3^\rho
 C(\eta_1,\eta_2)\leq\eta_2.
\]
Chosen $\eta_3$ is sufficiently small depending on $\eta_1$ and
$\eta_2$. By using H\"older, Hardy-Littlewood-Sobolev inequality,
\eqref{5.27},\ \eqref{5.29} and \eqref{5.34}, we have
\begin{eqnarray*}
\parallel
P_{lo}((|x|^{-2}\ast|u|^2)u)\parallel_{\dot{N}^1(I\times\mathbb{R}^n)}&\lesssim&\parallel
u\parallel_{U(I)}^2\parallel\nabla u\parallel_{U(I)}\\
&\lesssim&\parallel u_{lo}\parallel_{U(I)}^2\parallel\nabla
u_{lo}\parallel_{U(I)}+\parallel
u_{hi}\parallel_{U(I)}^2\parallel\nabla
u_{hi}\parallel_{U(I)}\\
&&+\parallel u_{lo}\parallel_{U(I)}^2\parallel\nabla
u_{hi}\parallel_{U(I)}+\parallel
u_{hi}\parallel_{U(I)}^2\parallel\nabla
u_{lo}\parallel_{U(I)}\\
&\lesssim&\eta_1^2\parallel
u_{lo}\parallel_{\dot{S}^1(I\times\mathbb{R}^n)}+(\eta_2L(E))^2L(E)\\
&&+\eta_1^2L(E)+(\eta_2L(E))^2\parallel
u_{lo}\parallel_{\dot{S}^1(I\times\mathbb{R}^n)}.
\end{eqnarray*}
Then, $\parallel
u_{lo}\parallel_{\dot{S}^1(I\times\mathbb{R}^n)}\lesssim
E+\eta_2+(\eta_2L(E))^2L(E)+\eta_1^2L(E)+(\eta_1^2+(\eta_2L(E))^2)\parallel
u_{lo}\parallel_{\dot{S}^1(I\times\mathbb{R}^n)}$.\\
Taking $\eta_1$ and $\eta_2$ sufficiently small depending on $E$, we
can get
\[
\parallel
u_{lo}\parallel_{\dot{S}^1(I\times\mathbb{R}^n)}\lesssim E.
\]
Then, \eqref{5.32} holds. Of course, \eqref{5.33} can be obtained by
\eqref{5.32}, since $J=C(\eta_1)$.\\
At last, we show \eqref{5.31} is true. We write
$u_{lo}=P_{\leq\eta_2}u_{lo}+P_{\eta_2<\cdot<\eta_2^{-1}}u_{lo}$.\par
 In dimension $n\geq5$, by interpolation, Sobolev embedding,
 Bernstein inequality, \eqref{5.17} and \eqref{5.32}, we have
 \begin{eqnarray*}
\parallel
P_{\eta_2<\cdot<\eta_2^{-1}}u_{lo}\parallel_{W([t_j,T])}&\lesssim&\parallel
P_{\eta_2<\cdot<\eta_2^{-1}}u_{lo}\parallel_{L_t^{n+1}L_x^{\frac{2n(n+1)}{n^2-n-6}}([t_j,T]\times\mathbb{R}^n)}^c\parallel
P_{\eta_2<\cdot<\eta_2^{-1}}u_{lo}\parallel_{L_t^2L_x^{\frac{2n}{n-4}}([t_j,T]\times\mathbb{R}^n)}^{1-c}\\
&\lesssim&\parallel|\nabla|^{\frac{3}{n+1}}
P_{\eta_2<\cdot<\eta_2^{-1}}u_{lo}\parallel_{Z([t_j,T])}^c\parallel
u_{lo}\parallel_{\dot{S}^1([t_j,T]\times\mathbb{R}^n)}^{1-c}\\
&\lesssim&\eta_2^{-\frac{3}{n+1}}\parallel
u_{lo}\parallel_{Z([t_j,T])}^cE^{1-c}\\
&\lesssim&\eta_2^{-\frac{3}{n+1}}\eta_3^cE^{1-c}\\
&\leq&\eta_2,
\end{eqnarray*}
where $c=\frac{4(n+1)}{(n-1)(n+2)}$.\par
 In dimension $n=4$, by using interpolation, Sobolev embedding, Bernstein
 inequality, the conservation of energy and \eqref{5.17}, we get
\begin{eqnarray*}
\parallel
P_{\eta_2<\cdot<\eta_2^{-1}}u_{lo}\parallel_{W([t_j,T])}
&\lesssim&\parallel
P_{\eta_2<\cdot<\eta_2^{-1}}u_{lo}\parallel_{L_t^{5}L_x^{\frac{20}{3}}([t_j,T]\times\mathbb{R}^n)}^\frac{5}{6}
\parallel
P_{\eta_2<\cdot<\eta_2^{-1}}u_{lo}\parallel_{L_t^\infty L_x^{4}([t_j,T]\times\mathbb{R}^n)}^{\frac{1}{6}}\\
&\lesssim&\parallel|\nabla|^{\frac{3}{5}}
P_{\eta_2<\cdot<\eta_2^{-1}}u_{lo}\parallel_{Z([t_j,T])}^\frac{5}{6}E^\frac{1}{6}\\
&\lesssim&(\eta_2^{-\frac{3}{5}}\eta_3)^\frac{5}{6}E^{\frac{1}{6}}\\
&\leq&\eta_2.
\end{eqnarray*}\par
In dimension $n=3$, by using interpolation, Sobolev embedding,
Bernstein
 inequality, the conservation of energy and \eqref{5.17}, we get
\begin{eqnarray*}
\parallel
P_{\eta_2<\cdot<\eta_2^{-1}}u_{lo}\parallel_{W([t_j,T])}
&\lesssim&\parallel
P_{\eta_2<\cdot<\eta_2^{-1}}u_{lo}\parallel_{L_t^{4}L_x^{\infty}([t_j,T]\times\mathbb{R}^n)}^\frac{2}{5}
\parallel
P_{\eta_2<\cdot<\eta_2^{-1}}u_{lo}\parallel_{L_t^\infty L_x^{6}([t_j,T]\times\mathbb{R}^n)}^{\frac{3}{5}}\\
&\lesssim&\parallel(1+|\nabla|)^{\frac{3}{4}+\epsilon}
P_{\eta_2<\cdot<\eta_2^{-1}}u_{lo}\parallel_{Z([t_j,T])}^\frac{2}{5}E^\frac{3}{5}\\
&\lesssim&(\eta_2^{-\frac{3}{4}}\eta_3)^\frac{2}{5}E^{\frac{3}{5}}\\
&\leq&\eta_2.
\end{eqnarray*}
Hence, in all dimension $n\geq3$, we all have
\[
\parallel
P_{\eta_2<\cdot<\eta_2^{-1}}u_{lo}\parallel_{W([t_j,T])}\leq\eta_2
\]
By Sobolev embedding, Bernstein inequality and \eqref{5.30}, we have
\begin{eqnarray*}
\parallel P_{\leq\eta_2}u_{lo}\parallel_{W([t_j,T])}\lesssim\parallel\nabla
P_{\leq\eta_2}u_{lo}\parallel_{L_t^\frac{2(n+2)}{n-2}L_x^\frac{2n(n+2)}{n^2+4}([t_j,T]\times\mathbb{R}^n)}\lesssim\eta_2\parallel
u_{lo}\parallel_{L_t^\frac{2(n+2)}{n-2}L_x^\frac{2n(n+2)}{n^2+4}([t_j,T]\times\mathbb{R}^n)}.
\end{eqnarray*}\par
 In dimension $n=3$, by interpolation, \eqref{5.30} and the
 conservation of mass, we get
 \[
 \parallel
 P_{\leq\eta_2}u_{lo}\parallel_{W([t_j,T])}\lesssim\eta_2\parallel
 u_{lo}\parallel_{U([t_j,T])}^\frac{3}{5}\parallel
 u_{lo}\parallel_{L_t^\infty L_x^2([t_j,T]\times\mathbb{R}^n)}^\frac{2}{5}\lesssim\eta_2\eta_1^\frac{3}{5}M^\frac{2}{5}\leq\eta_2
 \]
provided $\eta_1$ is chosen sufficiently small depending on $M$.\par
 In dimension $n=4$, because of $
 L_t^\frac{2(n+2)}{n-2}L_x^\frac{2n(n+2)}{n^2+4}=U$, then
 \[
\parallel
 P_{\leq\eta_2}u_{lo}\parallel_{W([t_j,T])}\lesssim\eta_2\eta_1\leq\eta_2
 \]
 In dimension $n\geq5$, by interpolation, \eqref{5.29} and \eqref{5.30}
 \begin{eqnarray*}
\parallel
 P_{\leq\eta_2}u_{lo}\parallel_{W([t_j,T])}&\lesssim&\eta_2\parallel
u_{lo}\parallel_{U([t_j,T])}^\frac{6}{n+2}\parallel
u_{lo}\parallel_{L_tL_x^\frac{2n}{n-2}([t_j,T]\times\mathbb{R}^n)}^\frac{n-4}{n+2}\\
&\lesssim&\eta_2\eta_1^\frac{6}{n+2}\parallel
u_{lo}\parallel_{\dot{S}^0([t_j,T]\times\mathbb{R}^n)}^\frac{n-4}{n+2}\lesssim\eta_2\eta_1^\frac{6}{n+2}C(M)\leq\eta_2.
 \end{eqnarray*}
 Hence, in all dimension $n\geq3$, we get
 \[
\parallel
 P_{\leq\eta_2}u_{lo}\parallel_{W([t_j,T])}\leq\eta_2.
 \]
 Therefore, by the triangle inequality, \eqref{5.31} is true.
\end{proof}\par
Now, we are ready to show $\Omega_2\subset\Omega_1$. We will first
show \eqref{5.21}. The method is to compare $u_{lo}$ to $v$ via the
perturbation result of Lemma \ref{lemma3.3}. $u_{lo}$ satisfies the
following initial value problem on the slab
$[t_0,T]\times\mathbb{R}^n$
\begin{eqnarray*}
\left\{\begin{array}{ll}
           (i\partial_t+\Delta)u_{lo}=(|x|^{-2}\ast|u_{lo}|^2)u_{lo}+P_{lo}(|u|^\frac{4}{n-2}u)\\
           \hspace{3cm}+P_{lo}[(|x|^{-2}\ast|u|^2)u
           -(|x|^{-2}\ast|u_{lo}|^2)u_{lo}]-P_{hi}((|x|^{-2}\ast|u_{lo}|^2)u_{lo})\\
           u_{lo}(t_0)=u_{lo}(a).
           \end{array}
           \right.
\end{eqnarray*}
Since \eqref{5.30} and $v(t_0)=u_{lo}(t_0)$, in order to use Lemma
\ref{lemma3.3}, we only need to show the error term
\[
e=P_{lo}(|u|^\frac{4}{n-2}u)
           +P_{lo}[(|x|^{-2}\ast|u|^2)u
           -(|x|^{-2}\ast|u_{lo}|^2)u_{lo}]-P_{hi}((|x|^{-2}\ast|u_{lo}|^2)u_{lo})
\]
is small in $\dot{N}^0([t_0,T]\times\mathbb{R}^n)$.\par
 By using Lemma \ref{lemma},\ \eqref{5.17} and \eqref{5.28}, we have
 \[
 \parallel
 P_{lo}(|u|^\frac{4}{n-2}u)\parallel_{\dot{N}^0([t_0,T]\times\mathbb{R}^n)}\lesssim\parallel
 u\parallel_{Z([t_0,T])}^\theta\parallel
 u\parallel_{\dot{S}^1([t_0,T]\times\mathbb{R}^n)}^{\frac{n+2}{n-2}-\theta}\lesssim\eta_3^\theta(C(\eta_1,\eta_2))^{\frac{n+2}{n-2}-\theta}\leq\eta_2^{1-\delta}
 \]
provided $\eta_3$ is chosen sufficiently small depending on $\eta_1$
and $\eta_2$. By using Bernstein inequality, H\"older inequality,
Hardy-littlewood-Sobolev inequality, \eqref{5.30} and \eqref{5.33},
we have
\begin{eqnarray*}
\parallel
P_{hi}((|x|^{-2}\ast|u_{lo}|^2)u_{lo})\parallel_{\dot{N}^0([t_0,T]\times\mathbb{R}^n)}&\lesssim&\eta_2\parallel
\nabla
P_{hi}((|x|^{-2}\ast|u_{lo}|^2)u_{lo}\parallel_{\dot{N}^0([t_0,T]\times\mathbb{R}^n)}\\
&\lesssim&\eta_2\parallel
u_{lo}\parallel_{U([t_0,T])}^2\parallel\nabla
u_{lo}\parallel_{U([t_0,T])}\\
&\lesssim&\eta_2\parallel
u_{lo}\parallel_{\dot{S}^0([t_0,T]\times\mathbb{R}^n)}^2\parallel\nabla
u_{lo}\parallel_{\dot{S}^1([t_0,T]\times\mathbb{R}^n)}\\
&\lesssim&\eta_2C(M)C(\eta_1)E\\
&\leq&\eta_2^{1-\delta}
\end{eqnarray*}
provided $\eta_2$ is sufficiently small depending on $E,M$ and
$\eta_1$. From H\"older inequality, Hardy-littlewood-Sobolev
inequality, \eqref{5.30} and \eqref{5.35}, one can get
\begin{eqnarray*}
&&\parallel P_{lo}[(|x|^{-2}\ast|u|^2)u
           -(|x|^{-2}\ast|u_{lo}|^2)u_{lo}]\parallel_{\dot{N}^0([t_0,T]\times\mathbb{R}^n)}\\
          &\lesssim&\parallel(|x|^{-2}\ast|u_{lo}|^2)u_{hi}\parallel_{\dot{N}^0([t_0,T]\times\mathbb{R}^n)}
           \\
           &&+\parallel(|x|^{-2}\ast|u_{hi}|^2)u_{hi}\parallel_{\dot{N}^0([t_0,T]\times\mathbb{R}^n)}+
           \parallel(|x|^{-2}\ast|u_{hi}|^2)u_{lo}\parallel_{\dot{N}^0([t_0,T]\times\mathbb{R}^n)}\\
           &\lesssim&\parallel
 u_{lo}\parallel_{\dot{S}^0([t_0,T]\times\mathbb{R}^n)}^2 \parallel
 u_{hi}\parallel_{\dot{S}^0([t_0,T]\times\mathbb{R}^n)}\\
 &&+\parallel
 u_{hi}\parallel_{\dot{S}^0([t_0,T]\times\mathbb{R}^n)}^2 \parallel
 u_{lo}\parallel_{\dot{S}^0([t_0,T]\times\mathbb{R}^n)}+\parallel
 u_{hi}\parallel_{\dot{S}^0([t_0,T]\times\mathbb{R}^n)}^3\\
 &\lesssim&
 C(M)\eta_2C(\eta_1)L(E)+(\eta_2C(\eta_1)L(E))^2C(M)+(\eta_2C(\eta_1)L(E))^3\\
 &\leq&\eta_2^{1-\delta}.
\end{eqnarray*}
Therefore,
\[
\parallel
 e\parallel_{\dot{N}^0([t_0,T]\times\mathbb{R}^n)}\leq3\eta_2^{1-\delta}
\]
and hence, taking $\eta_2$ sufficiently small depending on $M$, we
can apply Lemma \ref{lemma3.3} to get
\[
\parallel
u_{lo}-v\parallel_{\dot{S}^0([t_0,T]\times\mathbb{R}^n)}\leq
C(M)\eta_2^{1-\delta}\leq\eta_2^{1-2\delta}.
\]
Thus \eqref{5.21} is true. Now we turn to prove \eqref{5.24} is
true. The idea is to compare $u_{hi}$ to the energy-critical
$\textsl{NLS}$
\begin{equation}\label{5.37}
 \left\{
\begin{array}{ll}
iw_t+\Delta w=|w|^{\frac{4}{n-2}}w\\
w(t_j)=u_{hi}(t_j)
\end{array}
\right.
\end{equation}
Then, citing the result in \cite{7,11,15}, we know \eqref{5.37} is
globally wellposed and
\begin{equation}
\parallel
w\parallel_{\dot{S}^1(\mathbb{R}\times\mathbb{R}^n)}\leq C(E)
\end{equation}
Using Lemma \ref{lemma 3.6} and \eqref{5.34}, we also get
\[
\parallel
w\parallel_{\dot{S}^0(\mathbb{R}\times\mathbb{R}^n)}\leq
C(E)\parallel u_{hi}(t_j)\parallel_{L_x^2}\lesssim\eta_2C(E)L(E).
\]
$u_{hi}$ satisfies the following initial value problem on the slab
$[t_j,T]\times\mathbb{R}^n$
\begin{eqnarray*}
\left\{\begin{array}{ll}
           (i\partial_t+\Delta)u_{hi}=|u_{hi}|^\frac{4}{n-2}u_{hi}+P_{hi}((|x|^{-2}\ast|u|^2)u)\\
           \hspace{3cm}+P_{hi}(|u|^\frac{4}{n-2}u-|u_{hi}|^\frac{4}{n-2}u_{hi})-P_{lo}(|u_{hi}|^\frac{4}{n-2}u_{hi}),\\
           u_{hi}(t_j)=u_{hi}(t_j).
           \end{array}
           \right.
\end{eqnarray*}
In order to use Lemma \ref{lemma3.4}, we only need to show the error
term
\[
e=P_{hi}((|x|^{-2}\ast|u|^2)u)
           +P_{hi}(|u|^\frac{4}{n-2}u-|u_{hi}|^\frac{4}{n-2}u_{hi})-P_{lo}(|u_{hi}|^\frac{4}{n-2}u_{hi})\\
\]
is small in $\dot{N}^1([t_j,T]\times\mathbb{R}^n)$.\par
 From H\"older, Hardy-Littlewood-Sobolev inequality, \eqref{5.27},\ \eqref{5.29},\ \eqref{5.32},\
 \eqref{5.35} and \eqref{5.36}, we have
\begin{eqnarray*}
\parallel
P_{hi}((|x|^{-2}\ast|u|^2)u\parallel_{\dot{N}^1([t_j,T]\times\mathbb{R}^n)}&\lesssim&\parallel
u\parallel_{U([t_j,T])}^2\parallel\nabla
u\parallel_{U([t_0,T])}\\
&\lesssim&\parallel
u_{hi}\parallel_{\dot{S}^0([t_j,T]\times\mathbb{R}^n)}^2\parallel
u_{hi}\parallel_{\dot{S}^1([t_j,T]\times\mathbb{R}^n)}\\
&&+\parallel
u_{lo}\parallel_{\dot{S}^0([t_j,T]\times\mathbb{R}^n)}^2\parallel
u_{lo}\parallel_{\dot{S}^1([t_j,T]\times\mathbb{R}^n)}\\
&&+\parallel
u_{lo}\parallel_{\dot{S}^0([t_j,T]\times\mathbb{R}^n)}^2\parallel
u_{hi}\parallel_{\dot{S}^1([t_j,T]\times\mathbb{R}^n)}\\
&&+\parallel
u_{hi}\parallel_{\dot{S}^0([t_j,T]\times\mathbb{R}^n)}^2\parallel
u_{lo}\parallel_{\dot{S}^1([t_j,T]\times\mathbb{R}^n)}\\
&\lesssim&(\eta_2C(\eta_1)L(E))^2C(\eta_1)L(E)+\eta_1^2E+\eta_1^2L(E)+(\eta_2L(E))^2E\\
&\leq&\eta_2
\end{eqnarray*}
if $\eta_2$ is sufficiently small depending on $E$ and $\eta_1$.\par
 By using Bernstein inequality, Lemma \ref{lemma},\ \eqref{5.17} and
 \eqref{5.28}, one has
 \begin{eqnarray*}
 \parallel
 P_{lo}(|u_{hi}|^\frac{4}{n-2}u_{hi})\parallel_{\dot{N}^1([t_j,T]\times\mathbb{R}^n)}&\lesssim&\eta_2^{-1}\parallel
 |u_{hi}|^\frac{4}{n-2}u_{hi}\parallel_{\dot{N}^0([t_j,T]\times\mathbb{R}^n)}\\
 &\lesssim&\eta_2^{-1}\parallel
 u\parallel_{Z([t_j,T])}^\theta\parallel
 u\parallel_{\dot{S}^1([t_j,T]\times\mathbb{R}^n)}^{\frac{n+2}{n-2}-\theta}\\
 &\lesssim&\eta_2^{-1}\eta_3^\theta C(\eta_1,\eta_2)\\
 &\leq&\eta_2
 \end{eqnarray*}
if $\eta_3$ is sufficiently small depending on $\eta_1$ and
$\eta_2$.\par
 Now, we'll estimate the last term $\parallel
 P_{hi}(|u|^\frac{4}{n-2}u-|u_{hi}|^\frac{4}{n-2}u_{hi})\parallel_{\dot{N}^1([t_j,T]\times\mathbb{R}^n)}$.
 Since the function $z\rightarrow
 |z|^\frac{4}{n-2}\frac{z^2}{|z|^2}$ is H\"older continuous
 of order $\frac{4}{n-2}$, then
 \begin{eqnarray}
 \parallel
 P_{hi}(|u|^\frac{4}{n-2}u-|u_{hi}|^\frac{4}{n-2}u_{hi})\parallel_{\dot{N}^1([t_j,T]\times\mathbb{R}^n)}
 &\lesssim&\parallel
|u|^\frac{4}{n-2}u-|u_{hi}|^\frac{4}{n-2}u_{hi}\parallel_{\dot{N}^1([t_j,T]\times\mathbb{R}^n)}\nonumber\\
&\lesssim&\parallel |u|^\frac{4}{n-2}\nabla
u-|u_{hi}|^\frac{4}{n-2}\nabla
u_{hi}\parallel_{\dot{N}^0([t_j,T]\times\mathbb{R}^n)}\nonumber\\
&&+\parallel |u|^\frac{4}{n-2}\nabla
u_{lo}\parallel_{\dot{N}^0([t_j,T]\times\mathbb{R}^n)}\nonumber\\
&&+\parallel
|u|^\frac{4}{n-2}\frac{u^2}{|u|^2}-|u_{hi}|^\frac{4}{n-2}\frac{u_{hi}^2}{|u_{hi}|^2}\parallel_{\dot{N}^0([t_j,T]\times\mathbb{R}^n)}\nonumber\\
&\lesssim&\parallel |u|^\frac{4}{n-2}\nabla
u_{lo}\parallel_{\dot{N}^0([t_j,T]\times\mathbb{R}^n)}\label{5.39}\\
&&+\parallel(|u|^\frac{4}{n-2}-|u_{hi}|^\frac{4}{n-2})\nabla
u_{hi}\parallel_{\dot{N}^0([t_j,T]\times\mathbb{R}^n)}\label{5.40}\\
&&+\parallel |u_{lo}|^\frac{4}{n-2}\nabla
u_{hi}\parallel_{\dot{N}^0([t_j,T]\times\mathbb{R}^n)}\label{5.41}.
 \end{eqnarray}
For \eqref{5.39}, from Remark \ref{remark2.1}, Bernstein inequality,
\eqref{5.17},\ \eqref{5.25} and \eqref{5.32}, we have
\begin{eqnarray*}
\parallel |u|^\frac{4}{n-2}\nabla
u_{lo}\parallel_{\dot{N}^0([t_j,T]\times\mathbb{R}^n)}
&\lesssim&\parallel
 u\parallel_{Z([t_j,T])}^\rho\parallel
 u\parallel_{S^1([t_j,T]\times\mathbb{R}^n)}^{\frac{4}{n-2}-\rho}\parallel
 \nabla u_{lo}\parallel_{S^1([t_j,T]\times\mathbb{R}^n)}\\
 &\lesssim&\eta_3^\rho C(\eta_1,\eta_2)\eta_2^{-1}\parallel
 u_{lo}\parallel_{S^1([t_j,T]\times\mathbb{R}^n)}\\
 &\leq&\eta_2
\end{eqnarray*}
if $\eta_3$ is chosen sufficiently small depending on $\eta_1$
and $\eta_2$.\\
 For \eqref{5.40}, when the dimension $3\leq n<6$, by using H\"older inequality, \eqref{5.27},\ \eqref{5.30} and
 \eqref{5.32}, we can get
\begin{eqnarray*}
&&\parallel(|u|^\frac{4}{n-2}-|u_{hi}|^\frac{4}{n-2})\nabla
u_{hi}\parallel_{\dot{N}^0([t_j,T]\times\mathbb{R}^n)}\\
&\lesssim&
\parallel(|u|^\frac{6-n}{n-2}u_{lo}\nabla
u_{hi}\parallel_{\dot{N}^0([t_j,T]\times\mathbb{R}^n)}\\
&\lesssim&\left(\parallel
 u_{hi}\parallel_{\dot{S}^1([t_j,T]\times\mathbb{R}^n)}^\frac{6-n}{n-2}+\parallel
 u_{lo}\parallel_{\dot{S}^1([t_j,T]\times\mathbb{R}^n)}^\frac{6-n}{n-2}\right)\parallel
  \nabla u_{hi}\parallel_{\dot{S}^0([t_j,T]\times\mathbb{R}^n)}\parallel
 u_{lo}\parallel_{W([t_j,T])}\\
 &\lesssim&(L(E)+E)^\frac{6-n}{n-2}\eta_2L(E)\\
 &\leq&\eta_2^\frac{1}{2}
\end{eqnarray*}
provided $\eta_2$ is chosen sufficiently small depending on $E$.\\
When the dimension $n\geq6$, notice the inequality $(a+b)^p\leq
a^p+b^p$ as $a,b\geq0,\ p\leq1$, \eqref{5.27} and \eqref{5.31}, we
have
\begin{eqnarray*}
&&\parallel(|u|^\frac{4}{n-2}-|u_{hi}|^\frac{4}{n-2})\nabla
u_{hi}\parallel_{\dot{N}^0([t_j,T]\times\mathbb{R}^n)}\\
&\lesssim&
\parallel|u_{lo}|^\frac{4}{n-2}\nabla
u_{hi}\parallel_{\dot{N}^0([t_j,T]\times\mathbb{R}^n)}\\
&\lesssim&\parallel
 u_{hi}\parallel_{\dot{S}^1([t_j,T]\times\mathbb{R}^n)}\parallel
 u_{lo}\parallel_{W([t_j,T])}^\frac{4}{n-2}\\
 &\lesssim&L(E)\eta_2^\frac{4}{n-2}\\
 &\leq&\eta_2^\frac{3}{n-2}.
\end{eqnarray*}
Then \eqref{5.41} has been estimated from the above by
$\eta_2^\frac{3}{n-2}$.\par
 Therefore
 \[
\parallel
e\parallel_{\dot{N}^1([t_j,T]\times\mathbb{R}^n)}\leq\eta_2+\eta_2^\frac{1}{2}+2\eta_2^\frac{3}{n-2}\leq\eta_2^\frac{3}{n}
 \]
and hence, taking $\eta_2$ sufficiently small depending on $E$, we
can apply Lemma \ref{lemma3.4} to get
\[
\parallel
u_{hi}-w\parallel_{\dot{S}^1([t_j,T]\times\mathbb{R}^n)}\lesssim\eta_1^c
\]
for a small constant $c>0$ depending only on the dimension $n$. So
we can obtain
\[
\parallel
u_{hi}\parallel_{\dot{S}^1([t_j,T]\times\mathbb{R}^n)}\leq\parallel
u_{hi}-w\parallel_{\dot{S}^1([t_j,T]\times\mathbb{R}^n)}+\parallel
w\parallel_{\dot{S}^1([t_j,T]\times\mathbb{R}^n)}\lesssim\eta_1^c+C(E)\leq
L(E)
\]
Choosing $L(E)$ is sufficiently large.\par
 Finally, \eqref{5.25} follows from
 \begin{eqnarray*}
\parallel
u\parallel_{S^1([t_0,T]\times\mathbb{R}^n)}&\leq&\parallel
u_{hi}\parallel_{S^1([t_0,T]\times\mathbb{R}^n)}+\parallel
u_{lo}\parallel_{S^1([t_0,T]\times\mathbb{R}^n)}\\
&\leq& C(M)+C(\eta_1)E+\eta_2C(\eta_1)L(E)+C(\eta_1)L(E)\\
&\leq& C(\eta_1,\eta_2).
 \end{eqnarray*}
This proves that $\Omega_2\subset\Omega_1$. Hence, by induction
\[
\parallel
u\parallel_{S^1(J_{k_0}\times\mathbb{R}^n)}\leq C(\eta_1,\eta_2)
\]
As $J_{k_0}$ is arbitrary and the total number of intervals $J_k$ is
$K=K(E,M,\eta_3)$, put  these bounds together we obtain
\[
\parallel
u\parallel_{S^1(\mathbb{R}\times\mathbb{R}^n)}\leq
C(\eta_1,\eta_2,\eta_3)=C(E,M).
\]
\subsection{Global bounds in the case: $p=\frac{4}{n-2},\ 2\leq\gamma<4$ with $\gamma<n$ and $\lambda_1\cdot\lambda_2<0$ or
$\frac{4}{n}\leq p<\frac{4}{n-2},\ \gamma=4$ with $\gamma<n$ and
$\lambda_1\cdot\lambda_2<0$} \ \quad The approaches for both cases
are the same, so we only prove the first case here. Without loss of
generality, let $|\lambda_1|=|\lambda_2|=1$.\par In this case, we'll
view $u$ the perturbation to the energy-critical problem
\begin{equation*}
 \left\{
\begin{array}{ll}
iw_t+\Delta w=|w|^{\frac{4}{n-2}}w\\
w(0)=u_{hi}(0)
\end{array}
\right.
\end{equation*}
which  is globally well-posedness by \cite{7,11,15} and
\begin{equation}\label{5.42}
\parallel
w\parallel_{\dot{S}^1(\mathbb{R}\times\mathbb{R}^n)}\leq C(E,M).
\end{equation}
By Lemma \ref{lemma 3.6}, \eqref{5.42} implies
\begin{equation}\label{5.43}
\parallel
w\parallel_{\dot{S}^0(\mathbb{R}\times\mathbb{R}^n)}\leq
C(E,M)\parallel u_0\parallel_{L_x^2}\leq C(E,M)M^\frac{1}{2}.
\end{equation}
\begin{definition}
$\dot{D}^0(I):=V(I)\cap U(I)\cap
L_T^\frac{2(n+2)}{n-2}L_x^\frac{2(n+2)}{n^2+4}$.
\end{definition}
It is easy to know that
\begin{eqnarray}
\parallel(|x|^{-\gamma}\ast|u|^2)u\parallel_{\dot{N}^k(I\times\mathbb{R}^n)}&\lesssim&\parallel
u\parallel_{\dot{D}^k(I)}\parallel
u\parallel_{\dot{D}^0(I)}^{4-\gamma}\parallel
u\parallel_{\dot{D}^1(I)}^{\gamma-2}\label{5.44}\\
\parallel
 |u|^{\frac{4}{n-2}}u\parallel_{\dot{N}^k(I\times\mathbb{R}^n)}&\lesssim&\parallel
u\parallel_{\dot{D}^1(I)}^{\frac{4}{n-2}}\parallel
u\parallel_{\dot{D}^k(I)},\label{5.45}
\end{eqnarray}
where $k=0,1$.\\
Split $\mathbb{R}$ into $J=J(E,M,\eta)$ subintervals
$I_j=[t_j,t_{j+1}]$ such that
\[
\parallel
u\parallel_{\dot{D}^1(I_j)}\sim\eta,
\]
where $\eta>0$ be a small constant to be chosen later.\par
 Moreover, choosing $M$ sufficiently small depending on $E$ and
 $\eta$, in view of \eqref{5.43}, we may assume
\[
\parallel
w\parallel_{\dot{S}^0(\mathbb{R}\times\mathbb{R}^n)}\leq\eta.
\]
Then, we get
\begin{equation}\label{5.46}
\parallel
u\parallel_{D^1(I_j)}\sim\eta.
\end{equation}
In fact, on each slab $I_j\times\mathbb{R}^n$, we have
\begin{eqnarray}\label{5.47}
\parallel
e^{i(t-t_j)\Delta}w(t_j)\parallel_{D^1(I_j)}\leq\parallel
w\parallel_{D^1(I_j)}+C\parallel
w\parallel_{D^1(I_j)}^\frac{n+2}{n-2}\leq\eta+C\eta^\frac{n+2}{n-2}\leq2\eta
\end{eqnarray}
if $\eta$ is sufficiently small.\par
 Let $I_0=[t_0,t_1]$. Since $w(t_0)=u(t_0)=u_0$, by using Strichartz
 estimates, \eqref{5.44},\ \eqref{5.45} and \eqref{5.47}, we have
\[
\parallel
u\parallel_{D^1(I_0)}\leq2\eta+C\parallel
w\parallel_{D^1(I_0)}^\frac{n+2}{n-2}+C\parallel
w\parallel_{D^1(I_0)}^3.
\]
By a standard continuity argument, this yields
\begin{equation}\label{5.48}
\parallel
u\parallel_{D^1(I_0)}\leq4\eta
\end{equation}
if $\eta$ is chosen sufficiently small.\par
 On the other way, from Strichartz estimates, \eqref{5.44},\ \eqref{5.45} and
 \eqref{5.48}, we have
\begin{eqnarray*}
\parallel
u\parallel_{\dot{D}^0(I_0)}&\lesssim&M^\frac{1}{2}+\parallel
u\parallel_{\dot{D}^1(I_0)}^\frac{4}{n-2}\parallel
u\parallel_{\dot{D}^0(I_0)}+\parallel
u\parallel_{\dot{D}^0(I_0)}^{5-\gamma}\parallel
u\parallel_{\dot{D}^1(I_0)}^{\gamma-2}\\
&\lesssim&M^\frac{1}{2}+\eta^\frac{4}{n-2}\parallel
u\parallel_{\dot{D}^0(I_0)}+\parallel
u\parallel_{\dot{D}^0(I_0)}^{5-\gamma}\eta^{\gamma-2}.
\end{eqnarray*}
Therefore, choosing $\eta$ sufficiently small and $\gamma<4$, we get
\[
\parallel
u\parallel_{\dot{D}^0(I_0)}\lesssim M^\frac{1}{2}.
\]
In order to apply Lemma \ref{lemma3.4}, we need to show the error
$(|x|^{-\gamma}\ast|u|^2)u$ is small on the norm
$\dot{N}^1(I_0\times\mathbb{R}^n)$. In fact, by
\[
\parallel(|x|^{-\gamma}\ast|u|^2)u\parallel_{\dot{N}^1(I_0\times\mathbb{R}^n)}\lesssim
\parallel
u\parallel_{\dot{D}^1(I_0)}^{\gamma-1}\parallel
u\parallel_{\dot{D}^0(I_0)}^{4-\gamma}\lesssim\eta^{\gamma-1}M^{2-\frac{\gamma}{2}}\leq
M^{\delta_0}
\]
for a small constant $\delta_0>0$. Then taking $M$ sufficiently
small depending on $E$ and $\eta$, by Lemma \ref{lemma3.4} we get
\[
\parallel
u-w\parallel_{\dot{S}^1(I_0\times\mathbb{R}^n)}\leq M^{c\delta_0}
\]
for a small constant $c>0$ that depends only on the dimension $n$.
Strichartz estimate implies
\begin{eqnarray}\label{5.49}
\parallel
e^{i(t-t_1)\Delta}(u(t_1)-w(t_1))\parallel_{\dot{S}^1(I_1\times\mathbb{R}^n)}\leq
M^{c\delta_0}.
\end{eqnarray}
Now, we turn to the interval $I_1=[t_1,t_2]$. By using Strichartz
estimate, \eqref{5.44},\ \eqref{5.45},\ \eqref{5.47} and
\eqref{5.49}, one can get
\begin{eqnarray*}
\parallel
u\parallel_{D^1(I_1)}&\leq&\parallel
e^{i(t-t_1)\Delta}u(t_1)\parallel_{\dot{D}^0(I_1)}+\parallel
e^{i(t-t_1)\Delta}(u(t_1)-w(t_1))\parallel_{\dot{D}^1(I_1)}\\
&&+\parallel e^{i(t-t_1)\Delta}w(t_1)\parallel_{\dot{D}^1(I_1)}
+C\parallel u\parallel_{D^1(I_1)}^\frac{n+2}{n-2}+C\parallel
u\parallel_{D^1(I_1)}^3\\
&\lesssim&M^\frac{1}{2}+M^{c\delta_0}+\eta+\parallel
u\parallel_{D^1(I_1)}^\frac{n+2}{n-2}+\parallel
u\parallel_{D^1(I_1)}^3.
\end{eqnarray*}
Choosing $\eta,M$ sufficiently small, by a standard continuity
argument, we obtain
\[
\parallel
u\parallel_{D^1(I_1)}\leq4\eta.
\]
Moreover, arguing as above, we also get
\[
\parallel
u\parallel_{\dot{D}^0(I_1)}\lesssim M^\frac{1}{2}.
\]
For $M$ sufficiently small, we can apply Lemma \ref{lemma3.4} to
obtain
\[
\parallel
u-w\parallel_{\dot{S}^1(I_1\times\mathbb{R}^n)}\leq M^{c\delta_1}
\]
for a small constant $0<\delta_1<\delta_0$.\par
 By using the induction argument, choosing $M$ smaller at every step, we obtain
 \[
\parallel
u\parallel_{D^1(I_j)}\leq4\eta.
 \]
Summing these estimates over all intervals $I_j$ and for the total
number of these intervals is $J=J(E,M,\eta)$, we get
\[
\parallel
u\parallel_{D^1(\mathbb{R})}\lesssim J\eta\leq C(E,M).
\]
By  using Strichartz estimate, \eqref{5.44} and \eqref{5.45}, we get
\[
\parallel
u\parallel_{S^1(\mathbb{R}\times\mathbb{R}^n)}\lesssim\parallel
u_0\parallel_{H_x^1}+\parallel
u\parallel_{D^1(\mathbb{R})}^\frac{n+2}{n-2}+\parallel
u\parallel_{D^1(\mathbb{R})}^3\lesssim M+E+C(E)\leq C(E,M).
\]
\subsection{Global bounds in the case: $\frac{4}{n}\leq p<\frac{4}{n-2},\ 2\leq\gamma<4$ with $\gamma<n$ and $\lambda_1\cdot\lambda_2<0$ or
$p=\frac{4}{n},\ \gamma=2$ and $\lambda_1,\lambda_2>0$} \ \quad The
approaches for both cases are similar with the subsection5.6, the
only differentia is to compare $u$ to the free Schr\"odinger
equation
\[
i\widetilde{u}_t+\Delta\widetilde{u}=0,\qquad \widetilde{u}(0)=u_0.
\]
By Strichartz estimate, the global solution $\widetilde{u}$ obeys
the spacetime estimates
\begin{eqnarray*}
\parallel
\widetilde{u}\parallel_{S^1(\mathbb{R}\times\mathbb{R}^n)}&\lesssim&\parallel
u_0\parallel_{\dot{H}_x^1}\leq C(E,M),\\
\parallel
\widetilde{u}\parallel_{S^0(\mathbb{R}\times\mathbb{R}^n)}&\lesssim&\parallel
u_0\parallel_{L_x^2}\lesssim M^\frac{1}{2}.
\end{eqnarray*}
At this time,we define
\begin{definition}
$\dot{D}^0(I):=V(I)\cap U(I)\cap
L_t^\frac{2(n+2)}{n-2}L_x^\frac{2(n+2)}{n^2+4}$.
\end{definition}
By the similar method of subsection5.6, it is not difficult to know
that
\[
\parallel
u\parallel_{S^1(\mathbb{R}\times\mathbb{R}^n)}\leq C(E,M).
\]
\subsection{Finite global Strichartz norms imply scattering}
\ \quad At last, we'll show that finite global Strichartz norms
imply scattering. For simplicity, we only construct the scattering
state in the positive time direction. Similar arguments can be used
to construct the scattering state in the negative time
direction.\par For $0<t<\infty$, define
\[
u_+(t)=u_0-i\int_0^te^{-is\Delta}\left(\lambda_1|u|^pu+\lambda_2(|x|^{-\gamma}\ast|u|^2)u\right)\,ds.
\]
Since $u\in S^1(\mathbb{R}\times\mathbb{R}^n)$,  Strichartz
estimates and Lemma \ref{lemma2.8} show that $u_+(t)\in H_x^1$ for
all $t\in \mathbb{R}^+$, and for $0<\tau<t$, we have
\begin{eqnarray*}
\parallel u_+(t)-u_+(\tau)\parallel_{H_x^1}&\lesssim&\parallel\int_\tau^t
e^{i(t-s)\Delta}\left(\lambda_1|u|^pu+\lambda_2(|x|^{-\gamma}\ast|u|^2)u\right)\,ds\parallel_{L_t^\infty H_x^1([\tau,t]\times\mathbb{R}^n)}\\
&\lesssim&\parallel u\parallel_{V([\tau,t])}^{2-\frac{(n-2)p}{2}}\parallel u\parallel_{W([\tau,t])}^{\frac{np}{2}-2}\parallel(1+|\nabla|)u\parallel_{V([\tau,t])}\\
&&+\parallel u\parallel_{U([\tau,t])}^{4-\gamma}\parallel
u\parallel_{L_t^6L_x^\frac{6n}{3n-8}([\tau,t]\times\mathbb{R}^n)}^{\gamma-2}\parallel(1+|\nabla|)u\parallel_{U([\tau,t])},
\end{eqnarray*}
and for $\varepsilon>0$, there exists $T_\varepsilon>0$ such that
\[
\parallel u_+(t)-u_+(\tau)\parallel_{H_x^1}\leq\varepsilon
\]
for any $t,\tau>T_\varepsilon$. Thus $u_+(t)$ converges to some
function $u_+$ in $H_x^1$ as $t\rightarrow+\infty$. In fact
\[
u_+:=u_0-i\int_0^\infty
e^{-is\Delta}\left(\lambda_1|u|^pu+\lambda_2(|x|^{-\gamma}\ast|u|^2)u\right)\,ds.
\]
At last, the scattering follows from
\begin{eqnarray*}
\parallel e^{-it\Delta}u(t)-u_+\parallel_{H_x^1}&=&\parallel\int_t^\infty
e^{-is\Delta}\left(\lambda_1|u|^pu+\lambda_2(|x|^{-\gamma}\ast|u|^2)u\right)\,ds\parallel_{H_x^1}\\
&=&\parallel\int_t^\infty
e^{i(t-s)\Delta}\left(\lambda_1|u|^pu+\lambda_2(|x|^{-\gamma}\ast|u|^2)u\right)\,ds\parallel_{H_x^1}\\
&\lesssim&\parallel u\parallel_{V([t,\infty))}^{2-\frac{(n-2)p}{2}}\parallel u\parallel_{W([t,\infty))}^{\frac{np}{2}-2}\parallel(1+|\nabla|)u\parallel_{V([t,\infty))}\\
&&+\parallel u\parallel_{U([t,\infty))}^{4-\gamma}\parallel
u\parallel_{L_t^6L_x^\frac{6n}{3n-8}([t,\infty)\times\mathbb{R}^n)}^{\gamma-2}\parallel(1+|\nabla|)u\parallel_{U([t,\infty))},
\end{eqnarray*}
because the right term obviously tends to $0$ as
$t\rightarrow+\infty$. The other properties follow from conservation
of mass and energy.
\section{Blowup results}
 \ \quad From the Theorem \ref{1.1}, we can find that there are still many
 regions where the global well-posedness holds need a few additional
 conditions, for example small energy and small mass. In this
 section, we'll show that on these regions, under suitable assumptions
 the solution of \eqref{1} will blow up in finite time. We follow
 the method of Glassey \cite{9}, which is essentially a convexity
 method. We consider the variance
 \[
f(t)=\int_{\mathbb{R}^n}|x|^2|u(t,x)|^2\,dx.
 \]
For strong $H_x^1-$solution $u$ to \eqref{1} with initial datum
$u_0\in\Sigma$, it is well known that $f\in
C^2(-\mbox{T}_{\mbox{min}},\,\mbox{T}_{\mbox{max}})$ and we
have(see, for example the Chapter 6 of \cite{19});
\begin{lemma}
For all $t\in(-\mbox{T}_{\mbox{min}},\,\mbox{T}_{\mbox{max}})$, we
have
\[
f'(t)=4Im\int\bar{u}x\cdot\nabla u\,dx
\]
and
\begin{equation}\label{6.1}
f''(t)=16E+\frac{4np-16}{p+2}\lambda_1\parallel
u\parallel_{L_x^{p+2}}^{p+2}+2\lambda_2(\gamma-2)\int(|x|^{-\gamma}\ast
|u|^2)|u|^2\,dx.
\end{equation}
\end{lemma}
\par
If we can find, for all
$t\in(-\mbox{T}_{\mbox{min}},\,\mbox{T}_{\mbox{max}})$ there exists
a constant $A$ such that: $f''(t)\leq A$, then we have
\begin{equation}\label{6.2}
\parallel xu\parallel_{L^2}^2\leq \theta(t)
\end{equation}
where
\[
\theta(t)=\parallel
x\varphi\parallel_{L^2}^2+4tIm\int\bar{\varphi}x\cdot\nabla
\varphi\,dx+\frac{1}{2}t^2A.
\]
If assume $A$ is negative, observe that $\theta(t)$ is a
second-degree polynomial, then $\theta(t)<0$ for $|t|$ large enough.
Since $\parallel xu\parallel_{L^2}^2\geq0$, we deduce from
\eqref{6.1} that both $\mbox{T}_{\mbox{min}}$ and
$\mbox{T}_{\mbox{max}}$ are finite. However, it is not a necessary
and sufficient condition so that $\theta(t)$ takes negative values
that $A$ is negative. A necessary and sufficient condition so that
$\theta(t)$ takes negative values is that
\[
8(Im\int\bar{\varphi}x\cdot\nabla \varphi\,dx)^2>A\parallel
x\varphi\parallel_{L^2}^2.
\]
But in many states, we can't get both $\mbox{T}_{\mbox{min}}$ and
$\mbox{T}_{\mbox{max}}$ are finite. People who are interested in it
can see the Chapter 6 of \cite{19}.\par
 In the following, we'll find the negative constant $A$ such that $f''(t)\leq
 A$:\\
 case (1): $\lambda_1<0,\ \lambda_2>0,\ \frac{4}{n}\leq p\leq\frac{4}{n-2},\
 0<\gamma\leq\frac{np}{2}$ and $E<0$.\\
 By using \eqref{6.1}, the conservation of energy and our assumption, we
 get
 \begin{eqnarray}
f''(t)&=&16E+(4np-16)\{E-\frac{1}{2}\parallel \nabla
u\parallel_{L^2}^2-\frac{\lambda_2}{4}\int(|x|^{-\gamma}\ast
|u|^2)|u|^2\,dx\}\nonumber\\
&&+2\lambda_2(\gamma-2)\int(|x|^{-\gamma}\ast
|u|^2)|u|^2\,dx\nonumber\\
&=&4npE-(2np-8)\parallel \nabla
u\parallel_{L^2}^2-(np-2\gamma)\lambda_2\int(|x|^{-\gamma}\ast
|u|^2)|u|^2\,dx\label{6.3}\\
&\leq&4npE.\nonumber
 \end{eqnarray}
Let $A:=4npE<0$, then we find the negative constant $A$.\\
case (2): $\lambda_1>0,\ \lambda_2<0,\ \frac{2\gamma}{n}\leq
p\leq\frac{4}{n-2},\
 2\leq\gamma\leq4$ and $E<0$.\\
From \eqref{6.1}, the conservation of energy and our assumption, we
 get
\begin{eqnarray}
f''(t)&=&16E+\frac{4np-16}{p+2}\lambda_1\parallel
u\parallel_{L_x^{p+2}}^{p+2}\nonumber\\
&&+8(\gamma-2)\{E-\frac{1}{2}\parallel \nabla
u\parallel_{L^2}^2-\frac{\lambda_1}{p+2}\parallel
u\parallel_{L_x^{p+2}}^{p+2}\}\nonumber\\
&=&8\gamma E-4(\gamma-2)\parallel \nabla
u\parallel_{L^2}^2+\frac{4np-8\gamma}{p+2}\lambda_1\parallel
u\parallel_{L_x^{p+2}}^{p+2}\label{6.4}\\
&\leq&8\gamma E.\nonumber
 \end{eqnarray}
Let $A:=8\gamma E<0$, then we find the negative constant $A$.\\
case (3): $\lambda_1<0,\ \lambda_2<0,\ \frac{4}{n}\leq
p\leq\frac{4}{n-2},\
 2\leq\gamma\leq4$ and $E<0$.\\
When $\gamma\geq\frac{np}{2}$, using \eqref{6.3} and our assumption,
we have
\[
f''(t)\leq 4npE.
\]
When $\gamma<\frac{np}{2}$, from \eqref{6.4} and our assumption, we
have
\[
f''(t)\leq8\gamma E.
\]
So we also find the negative constant $A$.\\
case (4) $\lambda_1<0,\ \lambda_2<0,\ 0<\gamma<2,\
\frac{4}{n}<p\leq\frac{4}{n-2}$ and $4npE+C(M)<0$\\
By using \eqref{4.2} and Young's inequality, we have, when
$\gamma<2$,
\[
\parallel \nabla
u\parallel_{L^2}^\gamma\leq\delta\parallel \nabla
u\parallel_{L^2}^2+C(\delta).
\]
From \eqref{6.3} and our assumption, we have
\[
f''(t)\leq4npE+[C(np-2\gamma)|\lambda_2|\delta-(2np-8)]\parallel
\nabla u\parallel_{L^2}^2+C(np-2\gamma)|\lambda_2|C(\delta)\parallel
u\parallel_{L^2}^{4-\gamma}.
\]
Choosing $\delta$ sufficiently small, then
\[
f''(t)\leq4npE+C(M).
\]
Let $A:=4npE+C(M)<0$, then we find the negative constant $A$.\\
case (5) $\lambda_1<0,\ \lambda_2<0,\ 2<\gamma\leq4,\
0<p<\frac{4}{n}$ and $8\gamma E+C(M)<0$.\\
From \eqref{4.1} and Young's inequality, we have, when
$p<\frac{4}{n}$,
\[
\parallel
\nabla u\parallel_{L^2}^\frac{np}{2}\leq\delta\parallel \nabla
u\parallel_{L^2}^2+C(\delta).
\]
From \eqref{6.4} and our assumption, we have
\[
f''(t)\leq8\gamma E-4(\gamma-2)\parallel \nabla
u\parallel_{L^2}^2+\left(\frac{4np-8\gamma}{p+2}\lambda_1C\delta\parallel
\nabla
u\parallel_{L^2}^2+\frac{4np-8\gamma}{p+2}\lambda_1C(\delta)\right)\parallel
u\parallel_{L^2}^\frac{4-(n-2)p}{2}.
\]
Choosing $\delta$ sufficiently small, then
\[
f''(t)\leq8\gamma E+C(M).
\]
Let $A:=8\gamma E+C(M)<0$, then we find the negative constant $A$.\\

\end{document}